\documentclass[11pt]{article}%
\usepackage{amsmath}
\usepackage{amsfonts}
\usepackage{amssymb}
\usepackage{graphicx}
\usepackage{version}
\usepackage[T1]{fontenc}
\usepackage{textcomp}
\usepackage{hyperref}%
\providecommand{\U}[1]{\protect\rule{.1in}{.1in}}
\newtheorem{theorem}{Theorem}[section]

\newtheorem{corollary}[theorem]{Corollary}

\newtheorem{definition}[theorem]{Definition}

\newtheorem{lemma}[theorem]{Lemma}

\newtheorem{proposition}[theorem]{Proposition}
\newtheorem{remark}[theorem]{Remark}

\newenvironment{proof}[1][Proof]{\noindent\textbf{#1.} }{\ \rule{0.5em}{0.5em}}
\begin{document}

\title{On the behavior of least energy solutions of a fractional $(p,q(p))$-Laplacian
problem as p goes to infinity}
\author{Grey Ercole$^{\text{\thinspace a}}$, Aldo H. S. Medeiros$^{\text{\thinspace
a}}$\thinspace\ and Gilberto A. Pereira$^{\text{\thinspace b}}\medskip$\\{\small {$^{\mathrm{a}}$} Universidade Federal de Minas Gerais, Belo
Horizonte, MG, 30.123-970, Brazil. }\\{\small {$^{\mathrm{b}}$} Universidade Federal de Ouro Preto, Ouro Preto, MG,
35.400-000, Brazil.}\\{\small E-mails: grey@mat.ufmg.br\thanks{Corresponding author} ,
aldomedeiros@ufmg.br and gilberto.pereira@ufop.edu.br }}
\maketitle

\begin{abstract}
We study the behavior as \ $p\rightarrow\infty$ of $u_{p},$ a positive least
energy solution of the problem
\[
\left\{
\begin{array}
[c]{lll}%
\left[  \left(  -\Delta_{p}\right)  ^{\alpha}+\left(  -\Delta_{q(p)}\right)
^{\beta}\right]  u=\mu_{p}\left\Vert u\right\Vert _{\infty}^{p-2}%
u(x_{u})\delta_{x_{u}} & \mathrm{in} & \Omega\\
u=0 & \mathrm{in} & \mathbb{R}^{N}\setminus\Omega\\
\left\vert u(x_{u})\right\vert =\left\Vert u\right\Vert _{\infty}, &  &
\end{array}
\right.
\]
where $\Omega\subset\mathbb{R}^{N}$ is a bounded, smooth domain,
$\delta_{x_{u}}$ is the Dirac delta distribution supported at $x_{u},$
\[
\lim_{p\rightarrow\infty}\frac{q(p)}{p}=Q\in\left\{
\begin{array}
[c]{lll}%
(0,1) & \mathrm{if} & 0<\beta<\alpha<1\\
(1,\infty) & \mathrm{if} & 0<\alpha<\beta<1
\end{array}
\right.
\]
and
\[
\lim_{p\rightarrow\infty}\sqrt[p]{\mu_{p}}>R^{-\alpha},
\]
with $R$ denoting the inradius of $\Omega.$

\end{abstract}


\noindent\textbf{2010 AMS Classification.} 35D40, 35R11, 35J60.

\section{Introduction}

Let $\Omega$ be a bounded, smooth domain of $\mathbb{R}^{N},$ $N>1,$ and
consider the Sobolev space of fractional order $s\in(0,1)$ and exponent
$m>1,$
\[
W_{0}^{s,m}(\Omega):=\left\{  u\in L^{m}(\mathbb{R}^{N}):u=0\ \mathrm{in}%
\ \mathbb{R}^{N}\setminus\Omega\quad\mathrm{and}\quad\left[  u\right]
_{s,m}<\infty\right\}  ,
\]
where
\[
\left[  u\right]  _{s,m}:=\left(  \int_{\mathbb{R}^{N}}\int_{\mathbb{R}^{N}%
}\frac{\left\vert u(x)-u(y)\right\vert ^{m}}{\left\vert x-y\right\vert
^{N+sm}}\mathrm{d}x\mathrm{d}y\right)  ^{\frac{1}{m}}%
\]
is the Gagliardo seminorm.

As it is well known, $\left(  W_{0}^{s,m}(\Omega),\left[  \cdot\right]
_{s,m}\right)  $ is a uniformly convex Banach space (also characterized as the
closure of $C_{c}^{\infty}(\Omega)$ with respect to $\left[  \cdot\right]
_{s,m}$), compactly embedded into $L^{r}(\Omega)$ whenever%
\[
1\leq r<m_{s}^{\ast}:=\left\{
\begin{array}
[c]{ll}%
\dfrac{Nm}{N-sm}, & m<N/s,\\
\infty, & m\geq N/s.
\end{array}
\right.
\]

Moreover,
\begin{equation}
W_{0}^{s,m}(\Omega)\hookrightarrow\hookrightarrow C_{0}(\overline{\Omega
})\quad\mathrm{if}\quad m>N/s. \label{cpt}%
\end{equation}
(The notation $A\hookrightarrow\hookrightarrow B$ means that the continuous
embedding $A\hookrightarrow B$ is compact.) It follows that infimum%
\[
\lambda_{s,m}:=\inf\left\{  \frac{\left[  u\right]  _{s,m}^{m}}{\left\Vert
u\right\Vert _{\infty}^{m}}:u\in W_{0}^{s,m}(\Omega)\setminus\left\{
0\right\}  \right\}
\]
is positive and, in fact, a minimum.

The compactness in (\ref{cpt}) is consequence of the following Morrey's type
inequality (see \cite{Guide})%
\begin{equation}
\sup_{(x,y)\not =(0,0)}\frac{\left\vert u(x)-u(y)\right\vert }{\left\vert
x-y\right\vert ^{s-\frac{N}{m}}}\leq C\left[  u\right]  _{s,m},\quad
\forall\,u\in W_{0}^{s,m}(\Omega) \label{Morrey}%
\end{equation}
which holds whenever $m>N/s.$ If $m$ is sufficiently large, the positive
constant $C$ in (\ref{Morrey}) can be chosen uniform with respect to $m$ (see
\cite[Remark 2.2]{FPL}).

Let $\left(  -\Delta_{m}\right)  ^{s}$ be the $s$-fractional $m$-Laplacian,
the operator acting from $W_{0}^{s,m}(\Omega)$ into its topological dual,
defined by
\[
\big\langle \left(  -\Delta_{m}\right)  ^{s}u,\varphi\big\rangle
_{s,m}:=\int_{\mathbb{R}^{N}}\int_{\mathbb{R}^{N}}\dfrac{\left\vert
u(x)-u(y)\right\vert ^{m-2}(u(x)-u(y))(\varphi(x)-\varphi(y))}{\left\vert
x-y\right\vert ^{N+sm}}\mathrm{d}x\mathrm{d}y,\quad\forall\,u,\varphi\in
W_{0}^{s,m}(\Omega).
\]
We recall that $\left(  -\Delta_{m}\right)  ^{s}u$ is the G\^{a}teaux
derivative at a function $u\in W_{0}^{s,m}(\Omega)$ of the Fr\'{e}chet
differentiable functional $v\mapsto m^{-1}\left[  v\right]  _{s,m}^{m}.$

An alternative pointwise expression for $\left(  -\Delta_{m}\right)  ^{s}u$
is
\begin{equation}
(\mathcal{L}_{s,m}u)(x):=2\int_{\mathbb{R}^{N}}\dfrac{\left\vert
u(x)-u(y)\right\vert ^{m-2}(u(y)-u(x))}{\left\vert x-y\right\vert ^{N+sm}%
}\mathrm{d}y. \label{Lsm}%
\end{equation}
As argued in \cite{LL}, this expression appears formally as follows
\begin{align*}
\big\langle\left(  -\Delta_{m}\right)  ^{s}u,\varphi\big\rangle_{s,m}  &
=\int_{\mathbb{R}^{N}}\int_{\mathbb{R}^{N}}\dfrac{\left\vert
u(x)-u(y)\right\vert ^{m-2}(u(x)-u(y))(\varphi(x)-\varphi(y))}{\left\vert
x-y\right\vert ^{N+sm}}\mathrm{d}x\mathrm{d}y\\
&  =\int_{\mathbb{R}^{N}}\varphi(x)\left(  \int_{\mathbb{R}^{N}}%
\dfrac{\left\vert u(x)-u(y)\right\vert ^{m-2}(u(x)-u(y))}{\left\vert
x-y\right\vert ^{N+sm}}\mathrm{d}y\right)  \mathrm{d}x\\
&  -\int_{\mathbb{R}^{N}}\varphi(y)\left(  \int_{\mathbb{R}^{N}}%
\dfrac{\left\vert u(x)-u(y)\right\vert ^{m-2}(u(x)-u(y))}{\left\vert
x-y\right\vert ^{N+sm}}\mathrm{d}x\right)  \mathrm{d}y\\
&  =\int_{\mathbb{R}^{N}}\varphi(x)(\mathcal{L}_{s,m}u)(x)\mathrm{d}x.
\end{align*}

In Section \ref{Sec1}, we consider the nonhomogeneous problem
\begin{equation}
\left\{
\begin{array}
[c]{lll}%
\left[  \left(  -\Delta_{p}\right)  ^{\alpha}+\left(  -\Delta_{q}\right)
^{\beta}\right]  u=\mu\left\vert u(x_{u})\right\vert ^{p-2}u(x_{u}%
)\delta_{x_{u}} & \mathrm{in} & \Omega\\
u=0 & \mathrm{in} & \mathbb{R}^{N}\setminus\Omega,\\
\left\vert u(x_{u})\right\vert =\left\Vert u\right\Vert _{\infty} &  &
\end{array}
\right.  \label{pq}%
\end{equation}
where $\alpha,\beta,$ $p,q$ and $\mu>0$ satisfy suitable conditions, $x_{u}%
\in\Omega$ is a point where $u$ attains its sup norm ($\left\vert
u(x_{u})\right\vert =\left\Vert u\right\Vert _{\infty}$), and $\delta_{x_{u}}$
is the Dirac delta distribution supported at $x_{u}.$

Proceeding as in \cite{AEP} and \cite{EP16}, one can arrive at (\ref{pq}) as
the limit case, as $r\rightarrow\infty,$ of the problem%
\[
\left\{
\begin{array}
[c]{lll}%
\left[  \left(  -\Delta_{p}\right)  ^{\alpha}+\left(  -\Delta_{q}\right)
^{\beta}\right]  u=\mu\left\Vert u\right\Vert _{r}^{p-r}\left\vert
u\right\vert ^{r-2}u & \mathrm{in} & \Omega\\
u=0 & \mathrm{in} & \mathbb{R}^{N}\setminus\Omega
\end{array}
\right.
\]
where $\left\Vert \cdot\right\Vert _{r}$ denotes the standard norm in the
Lebesgue space $L^{r}(\Omega).$

As usual, we interpret (\ref{pq}) as an identity between functionals applied
to the (weak) solution $u.$ Thus,
\begin{equation}
\big\langle\left(  -\Delta_{p}\right)  ^{\alpha}u,\varphi\big\rangle_{\alpha
,p}+\big\langle\left(  -\Delta_{q}\right)  ^{\beta}u,\varphi\big\rangle_{\beta
,q}=\mu\left\vert u(x_{u})\right\vert ^{p-2}u(x_{u})\varphi(x_{u})\quad
\forall\,\varphi\in X(\Omega), \label{weak1}%
\end{equation}
where $X(\Omega)$ is an appropriate Sobolev space (that will be derived in the
sequence). The functional at the left-hand side of (\ref{weak1}) is the
G\^{a}teaux derivative of the Fr\'{e}chet differentiable functional $v\mapsto
p^{-1}\left[  v\right]  _{\alpha,p}^{p}+q^{-1}\left[  v\right]  _{\beta,q}%
^{q}$ at $u.$ However, the functional at the right-hand side is merely related
to the right-sided G\^{a}teaux derivative of the functional $\varphi\mapsto
p^{-1}\left\Vert \varphi\right\Vert _{\infty}^{p}$ whenever $u$ assumes its
sup norm at a unique point $x_{u}$. This has to do with the following fact
(see Lemma \ref{gateaux} and Remark \ref{gateaux1}): if $u\in C(\overline
{\Omega})$ assumes its sup norm only at $x_{u}\in\Omega,$ then
\[
\lim_{\epsilon\rightarrow0^{+}}\frac{\left\Vert u+\epsilon\varphi\right\Vert
_{\infty}^{p}-\left\Vert u\right\Vert _{\infty}^{p}}{p\epsilon}=\left\vert
u(x_{u})\right\vert ^{p-2}u(x_{u})\varphi(x_{u}),\quad\forall\,\varphi\in
C(\overline{\Omega}).
\]

Therefore, we define the formal energy functional associated with (\ref{pq})
by
\[
E_{\mu}(u):=\frac{1}{p}\left[  u\right]  _{\alpha,p}^{p}+\frac{1}{q}\left[
u\right]  _{\beta,q}^{q}-\frac{\mu}{p}\left\Vert u\right\Vert _{\infty}%
^{p},\quad\mu>0,
\]
and formulate our hypotheses on $\alpha,\beta,$ $p$ and $q$ to guarantee the
well-definiteness of this functional. For this, we take into account
(\ref{cpt}) and the following known facts:

\begin{itemize}
\item $W_{0}^{s,p}(\Omega)\not \hookrightarrow W_{0}^{s,q}(\Omega)$ for any
$0<s<1\leq q<p\leq\infty$ (see \cite[Theorem 1.1]{PW}),

\item $W_{0}^{s_{2},m_{2}}(\Omega)\hookrightarrow W_{0}^{s_{1},m_{1}}%
(\Omega),$ whenever $0<s_{1}<s_{2}<1\leq m_{1}<m_{2}<\infty$ (see \cite[Lemma
2.6]{Brasco}).
\end{itemize}

Thus, we assume that $\alpha,\beta,$ $p$ and $q$ satisfy one of the following
conditions:%
\begin{equation}
0<\alpha<\beta<1\quad\mathrm{and}\quad N/\alpha<p<q \label{H1b}%
\end{equation}
or%

\begin{equation}
0<\beta<\alpha<1\quad\mathrm{and}\quad N/\beta<q<p. \label{H1a}%
\end{equation}

The assumption (\ref{H1b}) provides the chain of embeddings $W_{0}^{\beta
,q}(\Omega)\hookrightarrow W_{0}^{\alpha,p}(\Omega)\hookrightarrow
\hookrightarrow C_{0}(\overline{\Omega})$ whereas (\ref{H1a}) yields
$W_{0}^{\alpha,p}(\Omega)\hookrightarrow W_{0}^{\beta,q}(\Omega
)\hookrightarrow\hookrightarrow C_{0}(\overline{\Omega}).$ Therefore, the
Sobolev space
\[
X(\Omega):=\left\{
\begin{array}
[c]{lll}%
\left(  W_{0}^{\beta,q}(\Omega),\left[  \cdot\right]  _{\beta,q}\right)  &
\mathrm{if} & 0<\alpha<\beta<1\quad\mathrm{and}\quad N/\alpha<p<q\\
\left(  W_{0}^{\alpha,p}(\Omega),\left[  \cdot\right]  _{\alpha,p}\right)  &
\mathrm{if} & 0<\beta<\alpha<1\quad\mathrm{and}\quad N/\beta<q<p,
\end{array}
\right.
\]
is the natural domain for the energy functional $E_{\mu}.$ Note that
\[
X(\Omega)\subset W_{0}^{\alpha,p}(\Omega)\cap W_{0}^{\beta,q}(\Omega
)\mathrm{\quad\mathrm{and}\quad}X(\Omega)\hookrightarrow\hookrightarrow
C_{0}(\overline{\Omega}).
\]

Once we have chosen $X(\Omega),$ a weak solution of (\ref{pq}) is defined (see
Definition \ref{weak}) by means of (\ref{weak1}).

As for the parameter $\mu,$ we assume that
\begin{equation}
\mu>\lambda_{\alpha,p} \label{lamb}%
\end{equation}
where%
\begin{equation}
\lambda_{\alpha,p}:=\inf\left\{  \frac{\left[  u\right]  _{\alpha,p}^{p}%
}{\left\Vert u\right\Vert _{\infty}^{p}}:u\in W_{0}^{\alpha,p}(\Omega
)\setminus\left\{  0\right\}  \right\}  =\frac{\left[  e\right]  _{\alpha
,p}^{p}}{\left\Vert e\right\Vert _{\infty}^{p}}>0 \label{ep}%
\end{equation}
for some function $e\in W_{0}^{\alpha,p}(\Omega)\setminus\left\{  0\right\}
.$ The existence of $e$ is a consequence of the compact embedding of
$W_{0}^{\alpha,p}(\Omega)$ into $C_{0}(\overline{\Omega})$ that holds in both
cases (\ref{H1b}) and (\ref{H1a}).

It turns out that (\ref{lamb}) is also a necessary condition for the existence
of weak solutions (see Remark \ref{nec}).

Assuming the above conditions on $\alpha,\beta,p,q$ and $\mu$ we show the
existence of at least one positive weak solution that minimizes the energy
functional either on $W_{0}^{\beta,q}(\Omega)\setminus\left\{  0\right\}  ,$
when (\ref{H1b}) holds, or on the following Nehari-type set%
\begin{equation}
\mathcal{N}_{\mu}:=\left\{  u\in W_{0}^{\alpha,p}(\Omega)\setminus\left\{
0\right\}  :\left[  u\right]  _{\alpha,p}^{p}+\left[  u\right]  _{\beta,q}%
^{q}=\mu\left\Vert u\right\Vert _{\infty}^{p}\right\}  , \label{Neh}%
\end{equation}
when (\ref{H1a}) holds. Both type of minimizers are referred in this work as
\textit{least energy solutions} of (\ref{pq}). The reason behind the
appearance of the Dirac delta is that the set where a minimizer of $E_{\mu}$
attains its sup norm is a singleton (as we will show).

We conclude Section \ref{Sec1} by observing that the weak solutions of
(\ref{pq}) are also viscosity solutions of
\[
\mathcal{L}_{\alpha,p}u+\mathcal{L}_{\beta,q}u=0\quad\mathrm{in}%
\,D:=\Omega\setminus\left\{  x_{u}\right\}
\]
and use this fact to argue that nonnegative least energy solutions are
strictly positive in $\Omega.$

In Section \ref{Sec2}, we fix the fractional orders $\alpha$ and $\beta$ (with
$\alpha\not =\beta$), allow $q$ and $\mu$ to depend suitably on $p$ ($q=q(p)$
and $\mu=\mu_{p}$) and denote by $u_{p}$ the positive least energy solution of
the problem%
\[
\left\{
\begin{array}
[c]{lll}%
\left[  \left(  -\Delta_{p}\right)  ^{\alpha}+\left(  -\Delta_{q(p)}\right)
^{\beta}\right]  u=\mu_{p}\left\vert u(x_{p})\right\vert ^{p-2}u(x_{p}%
)\delta_{x_{p}} & \mathrm{in} & \Omega\\
u=0 & \mathrm{in} & \mathbb{R}^{N}\setminus\Omega\\
\left\vert u(x_{p})\right\vert =\left\Vert u\right\Vert _{\infty}. &  &
\end{array}
\right.
\]
In the sequence we determine the asymptotic behavior of the pair $\left(
u_{p},x_{p}\right)  \in X(\Omega)\times\Omega,$ as $p$ goes to $\infty.$

Our main results are stated in Theorem \ref{Main} below, where, for each
$s\in(0,1],$
\[
C_{0}^{0,s}(\overline{\Omega}):=\left\{  u\in C_{0}(\overline{D}):\,\left\vert
u\right\vert _{s}<\infty\right\}  ,
\]
with $\left\vert \cdot\right\vert _{s}$ denoting the $s$-H\"{o}lder seminorm,
defined by%
\begin{equation}
\left\vert u\right\vert _{s}:=\sup\left\{  \frac{\left\vert
u(x)-u(y)\right\vert }{\left\vert x-y\right\vert ^{s}}:x,y\in\overline{\Omega
}\quad\mathrm{and}\quad x\not =y\right\}  . \label{bholder}%
\end{equation}

\begin{theorem}
\label{Main}Assume that%
\[
\lim_{p\rightarrow\infty}\frac{q(p)}{p}=:Q\in\left\{
\begin{array}
[c]{lll}%
(0,1) & \mathrm{if} & 0<\beta<\alpha<1\\
(1,\infty) & \mathrm{if} & 0<\alpha<\beta<1
\end{array}
\right.
\]
and%
\[
\Lambda:=\lim_{p\rightarrow\infty}\sqrt[p]{\mu_{p}}>R^{-\alpha},
\]
where $R$ is the inradius of $\Omega$ (i.e. the radius of the largest ball
inscribed in $\Omega$).

Let $p_{n}\rightarrow\infty.$ There exist $x_{\infty}\in\Omega$ and
$u_{\infty}\in C_{0}^{0,\beta}(\overline{\Omega})$ such that, up to a
subsequence, $x_{u_{p_{n}}}\rightarrow x_{\infty}$ and $u_{p_{n}}\rightarrow
u_{\infty}$ uniformly in $\overline{\Omega}.$ Moreover:
\end{theorem}

\begin{enumerate}
\item[(i)] $0<u_{\infty}(x)\leq\left(  \Lambda R^{\beta}\right)  ^{\frac
{1}{Q-1}}\left(  \operatorname{dist}(x,\partial\Omega)\right)  ^{\beta}%
\quad\forall\,x\in\Omega,$

\item[(ii)] $\operatorname{dist}(x_{\infty},\partial\Omega)=R,$

\item[(iii)] $u_{\infty}(x_{\infty})=\left\Vert u_{\infty}\right\Vert
_{\infty}=R^{\beta}(\Lambda R^{\beta})^{\frac{1}{Q-1}},$

\item[(iv)] $\left\vert u_{\infty}\right\vert _{\beta}=(\Lambda R^{\beta
})^{\frac{1}{Q-1}},$

\item[(v)] $\dfrac{\left\vert u_{\infty}\right\vert _{\beta}}{\left\Vert
u_{\infty}\right\Vert _{\infty}}=R^{-\beta}=\min\left\{  \dfrac{\left\vert
v\right\vert _{\beta}}{\left\Vert v\right\Vert _{\infty}}:v\in C_{0}^{0,\beta
}(\overline{\Omega})\setminus\left\{  0\right\}  \right\}  ,$

\item[(vi)] $u_{\infty}$ is a viscosity solution of
\[
\max\left\{  \mathcal{L}_{\alpha}^{+}u,\left(  \mathcal{L}_{\beta}%
^{+}u\right)  ^{Q}\right\}  =\max\left\{  -\mathcal{L}_{\alpha}^{-}u,\left(
-\mathcal{L}_{\beta}^{-}u\right)  ^{Q}\right\}  \quad\mathrm{in}%
\,\Omega\setminus\left\{  x_{\infty}\right\}  .
\]

\end{enumerate}

In the above equation the operators are defined according to the following
notation, where $0<s<1$:
\begin{equation}
\left(  \mathcal{L}_{s}^{+}u\right)  (x):=\sup_{y\in\mathbb{R}^{N}%
\setminus\left\{  x\right\}  }\frac{u(y)-u(x)}{\left\vert y-x\right\vert ^{s}%
}\quad\mathrm{and}\quad\left(  \mathcal{L}_{s}^{-}u\right)  (x):=\inf
_{y\in\mathbb{R}^{N}\setminus\left\{  x\right\}  }\frac{u(y)-u(x)}{\left\vert
y-x\right\vert ^{s}}. \label{Ls+-}%
\end{equation}

There are a substantial amount of papers in the recent literature dealing with
the asymptotic behavior of solutions as a parameter goes to infinity in
problems that involve a combination of first order, local operators and
nonlinearities of different homogeneity degrees (see \cite{AEP}, \cite{BM16},
\cite{ChaPe}, \cite{ChaPa}, \cite{ChaPa2}, \cite{SRS}, \cite{EP16},
\cite{MihaiRossi19}). In \cite{AEP}, Alves, Ercole and Pereira determined the
asymptotic behavior, as $p\rightarrow\infty,$ of the following problem of
order $1$
\begin{equation}
\left\{
\begin{array}
[c]{lll}%
\left[  -\Delta_{p}+(-\Delta_{q(p)})\right]  u=\mu_{p}\left\vert
u(x_{u})\right\vert ^{p-2}u(x_{u})\delta_{x_{u}} & \mathrm{in} & \Omega\\
u=0 & \mathrm{in} & \partial\Omega\\
\left\vert u(x_{u})\right\vert =\left\Vert u\right\Vert _{\infty}. &  &
\end{array}
\right.  \label{local}%
\end{equation}
Their work motived us to formulate an adequate fractional version of
(\ref{local}) and study, in the present paper, the behavior of the
corresponding least energy solutions as $p$ goes to infinity.

As for fractional operators, there are few works focusing such type of
asymptotic behavior. Most of recent ones deal with the problem of determining
the limit equation satisfied, in the viscosity sense, by the limit functions
(as $m\rightarrow\infty$) of a family $\left\{  u_{m}\right\}  $ of
minimizers. In general, such limit equation combines the operators
$\mathcal{L}_{s}^{+},$ $\mathcal{L}_{s}^{-}$ and their sum
\[
\mathcal{L}_{s}:=\mathcal{L}_{s}^{+}+\mathcal{L}_{s}^{-}.
\]

We refer to this latter operator as $s$-H\"{o}lder infinity Laplacian,
accordingly \cite{Chamb}, where it was introduced. In that paper, Chambolle,
Lindgren and Monneau studied the problem of minimizing the functional%
\[
\left[  u\right]  _{\Omega,s,m}:=\int_{\Omega}\int_{\Omega}\frac{\left\vert
u(x)-u(y)\right\vert ^{m}}{\left\vert x-y\right\vert ^{N+sm}}\mathrm{d}%
x\mathrm{d}y
\]
on the set
\[
X_{g}:=\left\{  u\in C(\overline{\Omega}):u=g\quad\mathrm{on}\,\partial
\Omega\right\}
\]
where $g\in C^{0,s}(\partial\Omega)$ is given. After showing the existence of
a unique minimizer $u_{m}\in X_{g}$ for this problem (assuming $m>N/s$), they
proved that, up to a subsequence, $u_{m}\rightarrow u_{\infty}\in
C^{0,s}(\overline{\Omega})$ uniformly and that this limit function is a
viscosity solution of
\[
\left\{
\begin{array}
[c]{lll}%
\mathcal{L}_{s}u=0 & \mathrm{in} & \Omega\\
u=g & \mathrm{on} & \partial\Omega.
\end{array}
\right.
\]
They also showed that $u_{\infty}$ is an optimal H\"{o}lder extension of $g$
in $\Omega.$

In \cite{LL}, Lindqvist and Lindgren characterized the asymptotic behavior (as
$m\rightarrow\infty$) of the only positive, normalized first eigenfunction
$u_{m}$ of $(-\Delta_{m})^{s}$ in $W_{0}^{s,m}(\Omega).$ That is, $u_{m}>0$ in
$\Omega,$ $\left\Vert u_{m}\right\Vert _{m}=1$ and $\left[  u_{m}\right]
_{s,m}^{m}=\Lambda_{s,m},$ where
\[
\Lambda_{s,m}:=\inf\left\{  \left[  u\right]  _{s,m}^{m}:u\in W_{0}%
^{s,m}(\Omega)\quad\mathrm{and}\quad\left\Vert u\right\Vert _{m}=1\right\}
\]
is the the first eigenvalue of $(-\Delta_{m})^{s}.$ Among several results,
they proved that%
\begin{equation}
\lim_{m\rightarrow\infty}\sqrt[m]{\Lambda_{s,m}}=R^{-s}\leq\frac{\left\vert
\phi\right\vert _{s}}{\left\Vert \phi\right\Vert _{\infty}}\quad\forall
\,\phi\in C_{c}^{\infty}(\Omega)\setminus\left\{  0\right\}  \label{aux2c}%
\end{equation}
and that any limit function $u_{\infty}$ of the family $\left\{
u_{m}\right\}  $ is a positive viscosity solution of the problem
\[
\left\{
\begin{array}
[c]{lll}%
\max\left\{  \mathcal{L}_{\infty}u\ ,\ \mathcal{L}_{\infty}^{-}u+R^{-s}%
u\right\}  =0 & \mathrm{in} & \Omega\\
u=0 & \mathrm{in} & \mathbb{R}^{N}\setminus\Omega.
\end{array}
\right.
\]

In \cite{FPL}, Ferreira and P\'{e}rez-Llanos studied the asypmtotic behavior,
as $m\rightarrow\infty,$ of the solutions of the problem%
\[
\left\{
\begin{array}
[c]{lll}%
{\displaystyle\int_{\mathbb{R}^{N}}}
\dfrac{\left\vert u(x)-u(y)\right\vert ^{m-2}(u(y)-u(x))}{\left\vert
x-y\right\vert ^{N+sm}}\mathrm{d}y=f(x,u) & \mathrm{in} & \Omega\\
u=g & \mathrm{in} & \mathbb{R}^{N}\setminus\Omega,
\end{array}
\right.
\]
for the cases $f=f(x)$ and $f=f(u)=\left\vert u\right\vert ^{\theta(m)-2}u$
with $\Theta:=\lim_{m\rightarrow\infty}\theta(m)/m<1$ (that is, the exponent
of the nonlinearity goes to infinity "sublinearly"). In the first case, they
obtained different limit equations involving the operators $\mathcal{L}%
_{\infty},$ $\mathcal{L}_{\infty}^{+}$ and $\mathcal{L}_{\infty}^{-}$
according to the sign of the function $f(x).$ In the second case, they
established the limit equation
\[
\min\left\{  -\mathcal{L}_{\infty}^{-}u-u^{\Theta},-\mathcal{L}_{\infty
}u\right\}  =0.
\]
Such results in that paper are compatible with the ones obtained for the local
operator in \cite{BDM} for the first case and in \cite{ChaPa} for the second case.

Recently, in \cite{SR}, Rossi and Silva studied the problem of minimizing the
Gagliardo seminorm $\left[  \cdot\right]  _{s,m}$ among the functions $v\in
W^{s,m}(\mathbb{R}^{N})$ satisfying the constraints
\begin{equation}
v=g\quad\mathrm{in}\,\mathbb{R}^{N}\setminus\Omega\quad\mathrm{and}%
\quad\mathcal{L}^{N}\left(  \left\{  v>0\right\}  \cap\Omega\right)
\leq\alpha,\label{vol}%
\end{equation}
where the function $g$ in $\mathbb{R}^{N}\setminus\Omega$ and the constant
$\alpha\in(0,\mathcal{L}^{N}(\Omega))$ are given, and $\mathcal{L}^{N}(D)$
denotes the $N$-dimensional Lebesgue volume of the subset $D\subset
\mathbb{R}^{N}.$

They proved that, up to subsequences, the family $\left\{  u_{m}\right\}  $ of
minimizers converges uniformly to a function $u_{\infty},$ as $m\rightarrow
\infty,$ that solves the equation%
\[
\mathcal{L}_{s}^{-}u=0\quad\mathrm{in}\,\left\{  u>0\right\}  \cap\Omega
\]
in the viscosity sense and also minimizes the $s$-H\"{o}lder seminorm
$\left\vert \cdot\right\vert _{s}$ among the functions in $W^{s,\infty
}(\mathbb{R}^{N})$ satisfying (\ref{vol}). Further, they showed the
convergence of the respective extremal values, that is: $\left[  u_{m}\right]
_{s,m}\rightarrow\left\vert u_{\infty}\right\vert _{s}.$

More recently, in \cite{EPS}, Ercole, Pereira and Sanchis studied the
asymptotic behavior of $u_{m},$ the positive solution of the minimizing
problem%
\[
\Lambda_{m}=\inf\left\{  \left[  u\right]  _{s,m}^{m}:u\in W_{0}^{s,m}%
(\Omega)\quad\mathrm{and}\quad\int_{\Omega}(\log\left\vert u\right\vert
)\omega\mathrm{d}x=0\right\}
\]
where $\omega\in L^{1}(\Omega)$ is a positive weight satisfying $\left\Vert
\omega\right\Vert _{1}=1.$ After showing that $u_{m}$ is the positive (weak)
solution of the singular problem
\[
\left\{
\begin{array}
[c]{lll}%
-(\Delta_{m})^{s}u=\Lambda_{m}\omega(x)u^{-1} & \mathrm{in} & \Omega\\
u=0 & \mathrm{in} & \mathbb{R}^{N}\setminus\Omega,
\end{array}
\right.
\]
they proved that, up to subsequence, $\left\{  u_{m}\right\}  $ converges
uniformly to a function $u_{\infty}\in C_{0}^{0,s}(\overline{\Omega})$ and
$\sqrt[m]{\Lambda_{m}}\rightarrow\left\vert u_{\infty}\right\vert _{s}.$
Moreover, the limit function $u_{\infty}$ is a positive viscosity solution of
\[
\left\{
\begin{array}
[c]{lll}%
\mathcal{L}_{s}^{-}u+\left\vert u\right\vert _{s}=0 & \mathrm{in} & \Omega\\
u=0 & \mathrm{in} & \mathbb{R}^{N}\setminus\Omega
\end{array}
\right.
\]
satisfying
\[
0\leq\int_{\Omega}(\log\left\vert u_{\infty}\right\vert )\omega\mathrm{d}%
x<\infty\quad\mathrm{and}\quad Q_{s}(u_{\infty})\leq Q_{s}(u)\quad
\forall\,u\in C_{0}^{0,s}(\overline{\Omega})\setminus\left\{  0\right\}  ,
\]
where $Q_{s}(u):=\left\vert u\right\vert _{s}/\exp\left(  \int_{\Omega}%
(\log\left\vert u\right\vert )\omega\mathrm{d}x\right)  .$

Our approach in this paper is inspired by the arguments and techniques
developed in some of the works above mentioned and can be applied to the
fractional version of \cite{EP16} and also for studying a fractional version
for the system considered in \cite{MihaiRossi19}.

\section{Existence of a positive least energy solution\label{Sec1}}

In this section, we assume that $\mu$ satisfy (\ref{lamb}) and that
$\alpha,\beta,p$ and $q$ are related by one of the conditions (\ref{H1b}) or
(\ref{H1a}). Our goal is to prove the existence of at least one positive least
energy solution $u_{\mu}\in X(\Omega)\setminus\left\{  0\right\}  $ for the
problem (\ref{pq}).

\begin{remark}
\label{nochange}We recall that $\left[  \left\vert u\right\vert \right]
_{s,p}\leq\left[  u\right]  _{s,p}$ for all $u\in W_{0}^{s,m}(\Omega)$ since
\[
\left\vert \left\vert u(x)\right\vert -\left\vert u(y)\right\vert \right\vert
<\left\vert u(x)-u(y)\right\vert \quad\mathrm{if}\,u(x)u(y)<0.
\]

\end{remark}

\begin{definition}
\label{weak}We say that a function $u\in X(\Omega)$ is a weak solution of
(\ref{pq}) if $\left\Vert u\right\Vert _{\infty}=\left\vert u(x_{u}%
)\right\vert $ and%
\[
\big\langle \left(  -\Delta_{p}\right)  ^{\alpha}u,\varphi\big\rangle
_{\alpha,p}+\big\langle \left(  -\Delta_{q}\right)  ^{\beta}u,\varphi
\big\rangle _{\beta,q}=\mu\left\vert u(x_{u})\right\vert ^{p-2}u(x_{u}%
)\varphi(x_{u}),\quad\forall\,\varphi\in X(\Omega).
\]

\end{definition}

\begin{remark}
\label{nec}If $u\in X(\Omega)$ is a weak solution of (\ref{pq}), then (by
taking $\varphi=u$)%
\[
\left[  u\right]  _{\alpha,p}^{p}+\left[  u\right]  _{\beta,q}^{q}%
=\mu\left\Vert u\right\Vert _{\infty}^{p}.
\]
If, in addition, $u\not \equiv 0$ the definition of $\lambda_{\alpha,p}$
yields
\[
\mu\left\Vert u\right\Vert _{\infty}^{p}=\left[  u\right]  _{\alpha,p}%
^{p}+\left[  u\right]  _{\beta,q}^{q}>\left[  u\right]  _{\alpha,p}^{p}%
\geq\lambda_{\alpha,p}\left\Vert u\right\Vert _{\infty}^{p}.
\]
This shows that (\ref{lamb}) is a necessary condition for the existence of a
nontrivial weak solution.
\end{remark}

\begin{proposition}
Suppose that $\alpha,\beta,p$ and $q$ satisfy (\ref{H1b}). There exists at
least one nonnegative function $u_{\mu}\in X(\Omega)\setminus\left\{
0\right\}  $ such that%
\[
E_{\mu}(u_{\mu})\leq E_{\mu}(u)\quad\forall\,u\in X(\Omega).
\]

\end{proposition}

\begin{proof}
Let
\begin{equation}
\lambda_{\beta,q}:=\inf\left\{  \frac{\left[  u\right]  _{\beta,q}^{q}%
}{\left\Vert u\right\Vert _{\infty}^{q}}:u\in W_{0}^{\beta,q}(\Omega
)\setminus\left\{  0\right\}  \right\}  >0. \label{lambbq}%
\end{equation}
Since $X(\Omega)=W_{0}^{\beta,q}(\Omega)$ we have
\[
E_{\mu}(u)\geq\frac{1}{q}\left[  u\right]  _{\beta,q}^{q}-\frac{\mu}%
{p}\left\Vert u\right\Vert _{\infty}^{p}\geq\frac{1}{q}\left[  u\right]
_{\beta,q}^{q}-\frac{\mu}{p}\left[  u\right]  _{\beta,q}^{p}\left(
\sqrt[q]{\lambda_{\beta,q}}\right)  ^{-p}=h(\left[  u\right]  _{\beta,q}%
)\quad\forall\,u\in X(\Omega),
\]
where $h:[0,\infty)\mapsto\mathbb{R}$ is given by%
\[
h(t):=\frac{1}{q}t^{q}-\frac{\mu}{p}\left(  \sqrt[q]{\lambda_{\beta,q}%
}\right)  ^{-p}t^{p}.
\]

Noting that $\lim_{t\rightarrow\infty}h(t)=\infty$ and
\[
h(t)\geq h(\left[  \mu\left(  \sqrt[q]{\lambda_{\beta,q}}\right)
^{-p}\right]  ^{\frac{1}{q-p}})=-\left(  \frac{1}{p}-\frac{1}{q}\right)
\left[  \mu\left(  \sqrt[q]{\lambda_{\beta,q}}\right)  ^{-p}\right]
^{\frac{q}{q-p}}%
\]
we conclude that $E_{\mu}$ is coercive and bounded from below. Hence, by
standards arguments of the Calculus of Variations (recall that $X(\Omega
)=W_{0}^{\beta,q}(\Omega)\hookrightarrow W_{0}^{\alpha,p}(\Omega
)\hookrightarrow\hookrightarrow C_{0}(\overline{\Omega})$) we can show that
the functional $E_{\mu}$ assumes the global minimum value at a function
$u_{\mu}\in X(\Omega).$

Now, in order to verify that $u_{\mu}\not \equiv 0$ we show that $E_{\mu
}(v)<0=E_{\mu}(0)$ for some $v\in W_{0}^{\beta,q}(\Omega).$ Let $u\in
W_{0}^{\alpha,p}(\Omega)\setminus\left\{  0\right\}  $ be such that
\[
\lambda_{\alpha,p}\leq\frac{\left[  u\right]  _{\alpha,p}^{p}}{\left\Vert
u\right\Vert _{\infty}^{p}}<\frac{\lambda_{\alpha,p}+\mu}{2}.
\]
By density and compactness, there exists a sequence $\left\{  \varphi
_{n}\right\}  \subset C_{c}^{\infty}(\Omega)$ such that $\left[  \varphi
_{n}\right]  _{\alpha,p}\rightarrow\left[  u\right]  _{\alpha,p}$ and
$\left\Vert \varphi_{n}\right\Vert _{\infty}\rightarrow\left\Vert u\right\Vert
_{\infty}.$ Therefore, there exists $n_{0}\in\mathbb{N}$ such that
\[
\frac{\left[  \varphi_{n_{0}}\right]  _{\alpha,p}^{p}}{\left\Vert
\varphi_{n_{0}}\right\Vert _{\infty}^{p}}\leq\frac{\lambda_{\alpha,p}+\mu}%
{2}<\mu.
\]
Since $\varphi_{n_{0}}\in X(\Omega)$ we have%
\[
E_{\mu}(t\varphi_{n_{0}})\leq\frac{t^{q}}{q}\left[  \varphi_{n_{0}}\right]
_{\beta,q}^{q}-\frac{t^{p}}{p}\frac{\left\Vert \varphi_{n_{0}}\right\Vert
_{\infty}^{p}}{2}\left(  \mu-\lambda_{\alpha,p}\right)  <0
\]
for some $t>0$ sufficiently small. Thus, $v:=t\varphi_{n_{0}}$ is such that
$E_{\mu}(v)<0.$

According to Remark \ref{nochange}, $E_{\mu}(\left\vert u_{\mu}\right\vert
)\leq E_{\mu}(u_{\mu}).$ Therefore, we can assume $u_{\mu}\geq0$ in $\Omega.$
\end{proof}

In the sequence we show that under (\ref{H1b}) any minimizer of the energy
functional $E_{\mu}$ is a weak solution of (\ref{pq}). For this we need the
following result proved in \cite{HL16}.

\begin{lemma}
\label{gateaux}Let $u\in C(\overline{\Omega})$ and $\Gamma_{u}:=\left\{
x\in\Omega:\left\vert u(x)\right\vert =\left\Vert u\right\Vert _{\infty
}\right\}  .$ Then,
\[
\lim_{\epsilon\rightarrow0^{+}}\frac{\left\Vert u+\epsilon\varphi\right\Vert
_{\infty}^{p}-\left\Vert u\right\Vert _{\infty}^{p}}{p\epsilon}=\max\left\{
\left\vert u(x)\right\vert ^{p-2}u(x)\varphi(x):x\in\Gamma_{u}\right\}
,\quad\forall\,\varphi\in C(\overline{\Omega}).
\]

\end{lemma}

\begin{remark}
\label{gateaux1}According to the notation of the Lemma \ref{gateaux}, if
\[
\max\left\{  \left\vert u(x)\right\vert ^{p-2}u(x)\varphi(x):x\in\Gamma
_{u}\right\}  \leq L\leq\min\left\{  \left\vert u(x)\right\vert ^{p-2}%
u(x)\varphi(x):x\in\Gamma_{u}\right\}  ,\quad\forall\,\varphi\in
C(\overline{\Omega}),
\]
for some $L\in\mathbb{R},$ then%
\[
\left\vert u(x)\right\vert ^{p-2}u(x)\varphi(x)=L=\left\vert u(y)\right\vert
^{p-2}u(y)\varphi(y),\quad\forall\,\varphi\in C(\overline{\Omega}%
)\quad\mathrm{and}\quad x,y\in\Gamma_{u}.
\]
Of course this implies that $\Gamma_{u}$ is a singleton, say $\Gamma
_{u}=\left\{  x_{u}\right\}  ,$ and therefore Lemma \ref{gateaux} yields
\[
\lim_{\epsilon\rightarrow0^{+}}\frac{\left\Vert u+\epsilon\varphi\right\Vert
_{\infty}^{p}-\left\Vert u\right\Vert _{\infty}^{p}}{p\epsilon}=L=\left\vert
u(x_{u})\right\vert ^{p-2}u(x_{u})\varphi(x_{u}),\quad\forall\,\varphi\in
C(\overline{\Omega}).
\]

\end{remark}

\begin{proposition}
\label{isweakb}Suppose that $\alpha,\beta,p$ and $q$ satisfy (\ref{H1b}). If
$u\in X(\Omega)$ satisfies
\[
E_{\mu}(u)\leq E_{\mu}(v)\quad\forall\,v\in X(\Omega),
\]
then $u$ is a weak solution of (\ref{pq}).
\end{proposition}

\begin{proof}
Let $\varphi\in X(\Omega)$ and $\epsilon>0.$ By hypothesis,
\[
0\leq\frac{E_{\mu}(u+\epsilon\varphi)-E_{\mu}(\varphi)}{\epsilon}%
=A(\epsilon)-\mu B(\epsilon)
\]
where%
\[
A(\epsilon):=\frac{\left[  u+\epsilon\varphi\right]  _{\alpha,p}^{p}-\left[
u\right]  _{\alpha,p}^{p}}{\epsilon p}+\frac{\left[  u+\epsilon\varphi\right]
_{\beta,q}^{q}-\left[  u\right]  _{\beta,q}^{q}}{\epsilon q}%
\]
and%
\[
B(\epsilon)=\frac{\left\Vert u+\epsilon\varphi\right\Vert _{\infty}%
^{p}-\left\Vert u\right\Vert _{\infty}^{p}}{\epsilon p}.
\]

As we already know (from the Introduction)
\[
L:=\lim_{\epsilon\rightarrow0^{+}}A(\epsilon)=\big\langle \left(  -\Delta
_{p}\right)  ^{\alpha}u,\varphi\big\rangle _{\alpha,p}+\big\langle \left(
-\Delta_{q}\right)  ^{\beta}u,\varphi\big\rangle _{\beta,q}.
\]

According to Lemma \ref{gateaux}%
\[
\lim_{\epsilon\rightarrow0^{+}}B(\epsilon)=\max\left\{  \left\vert
u(x)\right\vert ^{p-2}u(x)\varphi(x):x\in\Gamma_{u}\right\}  .
\]
Consequently,
\[
\mu\max\left\{  \left\vert u(x)\right\vert ^{p-2}u(x)\varphi(x):x\in\Gamma
_{u}\right\}  \leq L.
\]

Now, repeating the above arguments with $\varphi$ replaced with $-\varphi$ we
also conclude that%
\[
L\leq\mu\min\left\{  \left\vert u(x)\right\vert ^{p-2}u(x)\varphi
(x):x\in\Gamma_{u}\right\}  .
\]

It follows that (see Remark \ref{gateaux1}) $\Gamma_{u}=\left\{
x_{u}\right\}  $ and%
\[
L=\mu\left\vert u(x_{u})\right\vert ^{p-2}u(x_{u})\varphi(x_{u}).
\]

\end{proof}

Now, let us analyze $E_{\mu}$ under the hypothesis (\ref{H1a}). First we
observe that $E_{\mu}$ is unbounded from below in $X(\Omega).$ In fact, this
follows from the identity (where $e\in W_{0}^{\alpha,p}(\Omega)$ is given in
(\ref{ep}))
\begin{equation}
E_{\mu}(te)=\frac{t^{q}}{q}\left[  e\right]  _{\beta,q}^{q}-\frac{t^{p}}%
{p}\left(  \mu-\lambda_{\alpha,p}\right)  \left\Vert e\right\Vert _{\infty
}^{p},\quad\forall\,t>0. \label{Ee}%
\end{equation}

Thus, as usual, we look for a minimizer of $E_{\mu}$ restricted to Nehari-type
set $\mathcal{N}_{\mu}$ given by (\ref{Neh}).

Taking (\ref{H1a}) into account, the following properties for a function $u\in
X(\Omega)\setminus\left\{  0\right\}  $ can be easily verified
\begin{equation}
u\in\mathcal{N}_{\mu}\Longleftrightarrow E_{\mu}(u)=\left(  \frac{1}{q}%
-\frac{1}{p}\right)  \left[  u\right]  _{\beta,q}^{q} \label{1/q-1/p}%
\end{equation}
and
\begin{equation}
tu\in\mathcal{N}_{\mu}\Longleftrightarrow\left[  u\right]  _{\alpha,p}^{p}%
<\mu\left\Vert u\right\Vert _{\infty}^{p}\quad\mathrm{and}\quad t=\left(
\frac{\left[  u\right]  _{\beta,q}^{q}}{\mu\left\Vert u\right\Vert _{\infty
}^{p}-\left[  u\right]  _{\alpha,p}^{p}}\right)  ^{\frac{1}{p-q}}.
\label{condneh}%
\end{equation}

The latter property shows that $\mathcal{N}_{\mu}\not =\varnothing,$ since%
\[
\left[  e\right]  _{\alpha,p}^{p}=\lambda_{\alpha,p}\left\Vert e\right\Vert
_{\infty}^{p}<\mu\left\Vert e\right\Vert _{\infty}^{p}.
\]
Moreover, combining (\ref{lambbq}) and (\ref{1/q-1/p}) we obtain,
\[
\mu\left\Vert u\right\Vert _{\infty}^{p}=\left[  u\right]  _{\alpha,p}%
^{p}+\left[  u\right]  _{\beta,q}^{q}>\left[  u\right]  _{\beta,q}^{q}%
\geq\lambda_{\beta,q}\left\Vert u\right\Vert _{\infty}^{q},
\]
for an arbitrary $u\in\mathcal{N}_{\mu}.$ Consequently,%
\[
\left\Vert u\right\Vert _{\infty}>\left(  \frac{\lambda_{\beta,q}}{\mu
}\right)  ^{\frac{1}{p-q}}>0,\quad\forall\,u\in\mathcal{N}_{\mu}%
\]
and%
\[
E_{\mu}(u)\geq\left(  \frac{1}{q}-\frac{1}{p}\right)  \lambda_{\beta
,q}\left\Vert u\right\Vert _{\infty}^{q}>\left(  \frac{1}{q}-\frac{1}%
{p}\right)  \lambda_{\beta,q}\left(  \frac{\lambda_{\beta,q}}{\mu}\right)
^{\frac{q}{p-q}}>0,\quad\forall\,u\in\mathcal{N}_{\mu}.
\]

Another property is that%
\begin{equation}
\left[  u\right]  _{\alpha,p}<\frac{\sqrt[p]{\mu}}{\sqrt[q]{\lambda_{\beta,q}%
}}\left[  u\right]  _{\beta,q},\quad\forall\,u\in\mathcal{N}_{\mu},
\label{aux1}%
\end{equation}
which also follows from (\ref{lambbq}), since%
\[
\left[  u\right]  _{\alpha,p}^{p}<\left[  u\right]  _{\alpha,p}^{p}+\left[
u\right]  _{\beta,q}^{q}=\mu\left\Vert u\right\Vert _{\infty}^{p}\leq
\mu\left(  \frac{\left[  u\right]  _{\beta,q}^{q}}{\lambda_{\beta,q}}\right)
^{\frac{p}{q}}=\mu\left(  \lambda_{\beta,q}\right)  ^{-\frac{p}{q}}\left[
u\right]  _{\beta,q}^{p}.
\]

\begin{proposition}
Suppose that $\alpha,\beta,p$ and $q$ satisfy (\ref{H1a}). There exists at
least one nonnegative function $u_{\mu}\in W_{0}^{\alpha,p}(\Omega
)\setminus\left\{  0\right\}  $ such that%
\[
E_{\mu}(u_{\mu})\leq E_{\mu}(u)\quad\forall\,u\in\mathcal{N}_{\mu}.
\]

\end{proposition}

\begin{proof}
Let $\left\{  u_{n}\right\}  \in\mathcal{N}_{\mu}$ be a minimizing sequence:
\[
E_{\mu}(u_{n})=\left(  \frac{1}{q}-\frac{1}{p}\right)  \left[  u_{n}\right]
_{\beta,q}^{q}\rightarrow m_{\mu}:=\inf\left\{  E_{\mu}(u):u\in\mathcal{N}%
_{\mu}\right\}  .
\]

Taking (\ref{aux1}) into account and using compactness arguments, we can
assume that $u_{n}$ converges to a function $u_{\mu}\in W_{0}^{\alpha
,p}(\Omega)$ uniformly in $C(\overline{\Omega})$ and weakly in both Sobolev
spaces $W_{0}^{\alpha,p}(\Omega)$ and $W_{0}^{\beta,q}(\Omega).$ Of course,
$u_{\mu}\not \equiv 0$ since
\[
\left\Vert u_{\mu}\right\Vert _{\infty}>\left(  \frac{\lambda_{\beta,q}}{\mu
}\right)  ^{\frac{1}{p-q}}>0.
\]
Hence,%
\[
\left[  u_{\mu}\right]  _{\alpha,p}^{p}<\left[  u_{\mu}\right]  _{\alpha
,p}^{p}+\left[  u_{\mu}\right]  _{\beta,q}^{q}\leq\liminf_{n}\left(  \left[
u_{n}\right]  _{\alpha,p}^{p}+\left[  u_{n}\right]  _{\beta,q}^{q}\right)
=\mu\liminf_{n}\left\Vert u_{n}\right\Vert _{\infty}^{p}=\mu\left\Vert u_{\mu
}\right\Vert _{\infty}^{p},
\]
thus implying that $\theta u_{\mu}\in\mathcal{N}_{\mu},$ where%
\[
\theta:=\left(  \frac{\left[  u_{\mu}\right]  _{\beta,q}^{q}}{\mu\left\Vert
u_{\mu}\right\Vert _{\infty}^{p}-\left[  u_{\mu}\right]  _{\alpha,p}^{p}%
}\right)  ^{\frac{1}{p-q}}\leq1.
\]

Consequently,%
\begin{align*}
m_{\mu}  &  \leq E_{\mu}(\theta u_{\mu})\\
&  =\theta^{q}\left(  \frac{1}{q}-\frac{1}{p}\right)  \left[  u_{\mu}\right]
_{\beta,q}^{q}\leq\left(  \frac{1}{q}-\frac{1}{p}\right)  \liminf_{n}\left[
u_{n}\right]  _{\beta,q}^{q}=\lim_{n}E_{\mu}(u_{n})=m_{\mu},
\end{align*}
that is, $\theta=1,$ $u_{\mu}\in\mathcal{N}_{\mu}$ and $m_{\mu}=E_{\mu}%
(u_{\mu}).$

Remark \ref{nochange} and (\ref{condneh}) show that $\left\vert u_{\mu
}\right\vert \in\mathcal{N}_{\mu}$ and $E_{\mu}(\left\vert u_{\mu}\right\vert
)\leq E_{\mu}(u_{\mu}).$ Thus, we can assume that $u_{\mu}\geq0$ in $\Omega.$
\end{proof}

\begin{proposition}
Suppose that $\alpha,\beta,p$ and $q$ satisfy (\ref{H1a}). If $u\in
\mathcal{N}_{\mu}$ is such that
\[
E_{\mu}(u)\leq E_{\mu}(v),\quad\forall\,v\in\mathcal{N}_{\mu},
\]
then $u$ is a weak solution of (\ref{pq}).
\end{proposition}

\begin{proof}
Let $\varphi\in X(\Omega)$ be fixed. Since $u\in\mathcal{N}_{\mu}$ we have
$\mu\left\Vert u\right\Vert _{\infty}^{p}-\left[  u\right]  _{\alpha,p}%
^{p}=\left[  u\right]  _{\beta,q}^{q}>0.$ Thus, by continuity there exists
$\epsilon>0$ such that
\[
\mu\left\Vert u+s\varphi\right\Vert _{\infty}^{p}-\left[  u+s\varphi\right]
_{\alpha,p}^{p}>0,\quad\forall\,s\in(-\epsilon,\epsilon).
\]
It follows that
\[
\tau(s)(u+s\varphi)\in N_{\mu},\quad\forall\,s\in(-\epsilon,\epsilon)
\]
where%
\[
\tau(s):=\left(  \frac{\left[  u+s\varphi\right]  _{\beta,q}^{q}}%
{\mu\left\Vert u+s\varphi\right\Vert _{\infty}^{p}-\left[  u+s\varphi\right]
_{\alpha,p}^{p}}\right)  ^{\frac{1}{p-q}},\quad s\in(-\epsilon,\epsilon).
\]

Therefore, the function
\begin{align*}
\gamma(s)  &  :=E_{\mu}\left(  \tau(s)u+s\varphi\right) \\
&  =\frac{\tau(s)^{p}}{p}\left[  u+s\varphi\right]  _{\alpha,p}^{p}+\frac
{\tau(s)^{q}}{q}\left[  u+s\varphi\right]  _{\beta,q}^{q}-\mu\frac{\tau
(s)^{p}}{p}\left\Vert u+s\varphi\right\Vert _{\infty}^{p},\quad s\in
(-\epsilon,\epsilon)
\end{align*}
assumes a minimum value at $s=0.$ This implies that
\begin{equation}
\gamma^{\prime}(0^{+}):=\lim_{s\rightarrow0^{+}}\frac{\gamma(s)-\gamma(0)}%
{s}\geq0. \label{gamin}%
\end{equation}

Using Lemma \ref{gateaux} and observing that $\tau(0^{+})=1$ and
$u\in\mathcal{N}_{\mu}$ we compute%
\[
\gamma^{\prime}(0^{+})=\big\langle \left(  -\Delta_{p}\right)  ^{\alpha
}u,\varphi\big\rangle _{\alpha,p}+\big\langle \left(  -\Delta_{q}\right)
^{\beta}u,\varphi\big\rangle _{\beta,q}-\mu\max\left\{  \left\vert
u(x)\right\vert ^{p-2}u(x)\varphi(x):x\in\Gamma_{u}\right\}  .
\]

Hence, (\ref{gamin}) yields,
\[
\mu\max\left\{  \left\vert u(x)\right\vert ^{p-2}u(x)\varphi(x):x\in\Gamma
_{u}\right\}  \leq\big\langle \left(  -\Delta_{p}\right)  ^{\alpha}%
u,\varphi\big\rangle _{\alpha,p}+\big\langle \left(  -\Delta_{q}\right)
^{\beta}u,\varphi\big\rangle _{\beta,q}.
\]

Replacing $\varphi$ with $-\varphi$ we obtain
\[
\big\langle \left(  -\Delta_{p}\right)  ^{\alpha}u,\varphi\big\rangle
_{\alpha,p}+\big\langle \left(  -\Delta_{q}\right)  ^{\beta}u,\varphi
\big\rangle _{\beta,q}\leq\mu\min\left\{  \left\vert u(x)\right\vert
^{p-2}u(x)\varphi(x):x\in\Gamma_{u}\right\}  .
\]
Hence, according to Remark \ref{gateaux1}, $\Gamma_{u}=\left\{  x_{u}\right\}
$ and%
\[
\big\langle \left(  -\Delta_{p}\right)  ^{\alpha}u,\varphi\big\rangle
_{\alpha,p}+\big\langle \left(  -\Delta_{q}\right)  ^{\beta}u,\varphi
\big\rangle _{\beta,q}=\mu\left\vert u(x_{u})\right\vert ^{p-2}u(x_{u}%
)\varphi(x_{u}).
\]

\end{proof}

We gather the results above in the following theorem.

\begin{theorem}
\label{teo1}Suppose that $\alpha,\beta,p$ and $q$ satisfy either (\ref{H1b})
or (\ref{H1a}), and that $\mu$ satisfies (\ref{lamb}). Then (\ref{pq}) has at
least one nonnegative least energy solution $u_{\mu}\in X(\Omega
)\setminus\left\{  0\right\}  .$
\end{theorem}

We remark that $u_{\mu}\in X(\Omega)\setminus\left\{  0\right\}  $ given by
Theorem \ref{teo1} is a nonnegative weak solution of the fractional
harmonic-type equation
\begin{equation}
\left[  \left(  -\Delta_{p}\right)  ^{\alpha}+\left(  -\Delta_{q}\right)
^{\beta}\right]  u=0 \label{harpq}%
\end{equation}
in the punctured domain $\Omega\setminus\left\{  x_{u}\right\}  ,$ since
\begin{equation}
\big\langle \left(  -\Delta_{p}\right)  ^{\alpha}u_{\mu},\varphi
\big\rangle _{\alpha,p}+\big\langle \left(  -\Delta_{q}\right)  ^{\beta}%
u_{\mu},\varphi\big\rangle _{\beta,q}=0\quad\forall\,\varphi\in C_{c}^{\infty
}(\Omega\setminus\left\{  x_{u_{\mu}}\right\}  ). \label{harpq1}%
\end{equation}

Consequently, if $p>\frac{1}{1-\alpha}$ and $q>\frac{1}{1-\beta}$ (see Remark
\ref{+c}) one can adapt the arguments developed in \cite[Lemma 3.9]{FPL} and
\cite[Proposition 11]{LL} to verify that $u_{\mu}$ is also a viscosity
solution of
\begin{equation}
\mathcal{L}_{\alpha,p}u+\mathcal{L}_{\beta,q}u=0\quad\mathrm{in}%
\,\Omega\setminus\left\{  x_{u_{\mu}}\right\}  , \label{vispq}%
\end{equation}
(recall the definition of $\mathcal{L}_{s,m}$ in (\ref{Lsm})). This means that
$u_{\mu}$ is both a supersolution and a subsolution of (\ref{vispq}), that is,
$u_{\mu}$ meets the (respective) requirements:

\begin{itemize}
\item $(\mathcal{L}_{\alpha,p}\varphi)(x_{0})+(\mathcal{L}_{\beta,q}%
\varphi)(x_{0})\leq0$ for every pair $\left(  x_{0},\varphi\right)  \in
(\Omega\setminus\left\{  x_{u_{\mu}}\right\}  )\times C_{c}^{1}(\mathbb{R}%
^{N})$ satisfying
\[
\varphi(x_{0})=u_{\mu}(x_{0})\quad\mathrm{and}\quad\varphi(x)\leq u_{\mu
}(x)\quad\forall\,x\in\mathbb{R}^{N}\setminus\left\{  x_{u_{\mu}}%
,x_{0}\right\}  ,
\]

\item $(\mathcal{L}_{\alpha,p}\varphi)(x_{0})+(\mathcal{L}_{\beta,q}%
\varphi)(x_{0})\geq0$ for every pair $\left(  x_{0},\varphi\right)  \in
D\times C_{c}^{1}(\mathbb{R}^{N})$ satisfying%
\[
\varphi(x_{0})=u_{\mu}(x_{0})\quad\mathrm{and}\quad\varphi(x)\geq u_{\mu
}(x)\quad\forall\,x\in\mathbb{R}^{N}\setminus\left\{  x_{u_{\mu}}%
,x_{0}\right\}  .
\]

\end{itemize}

\begin{remark}
\label{+c}As observed in \cite{LL}, if $D$ is a bounded domain of
$\mathbb{R}^{N},$ $m>\frac{1}{s-1}$ and $\varphi\in C_{c}^{1}(\mathbb{R}%
^{N}),$ then the function $\mathcal{L}_{s,m}\varphi$ given by (\ref{Lsm}) is
well defined and continuous at each point $x\in D.$ Obviously, the same holds
for $\psi=\varphi+k,$ where $k$ is an arbitrary constant and $\varphi\in
C_{c}^{1}(\mathbb{R}^{N}),$ since
\[
\left(  \mathcal{L}_{s,m}\psi\right)  (x)=(\mathcal{L}_{s,m}\varphi)(x).
\]
Moreover, it is simple to check that $u_{\mu}$ fulfills both requirements
above even for test functions of the form $\psi=\varphi+k.$
\end{remark}

It is interesting to notice that $u_{\mu}>0$ in $\Omega\setminus\left\{
x_{u_{\mu}}\right\}  $ as consequence of $u_{\mu}$ being a supersolution of
(\ref{vispq}). The argument comes from \cite[Lemma 12]{LL}: by supposing that
$u_{\mu}(x_{0})=0$ for some $x_{0}\in\Omega\setminus\left\{  x_{u_{\mu}%
}\right\}  $ and noting that $0\not \equiv u\geq0,$ we can find a nonnegative
and nontrivial test function $\varphi\in C_{c}^{1}(\mathbb{R}^{N})$
satisfying
\[
\varphi(x_{0})=0\leq\varphi(x)\leq u_{\mu}(x)\quad\forall\,x\in\mathbb{R}%
^{N}\setminus\left\{  x_{u},x_{0}\right\}  .
\]
Hence,
\[
0\leq\int_{\mathbb{R}^{N}}\dfrac{2\left\vert \varphi(y)\right\vert ^{p-1}%
}{\left\vert x_{0}-y\right\vert ^{N+\alpha p}}\mathrm{d}y=\int_{\mathbb{R}%
^{N}}\dfrac{2\left\vert \varphi(y)\right\vert ^{p-2}\varphi(y)}{\left\vert
x_{0}-y\right\vert ^{N+\alpha p}}\mathrm{d}y\leq(\mathcal{L}_{\alpha,p}%
\varphi)(x_{0})+(\mathcal{L}_{\beta,q}\varphi)(x_{0})\leq0,
\]
which leads to the contradiction $\varphi\equiv0.$

\section{Asymptotic behavior as p goes to infinity \label{Sec2}}

Let $D$ be a bounded smooth (at least Lipschitz) domain of $\mathbb{R}^{N}.$
We recall that $\left(  C_{0}^{0,s}(\overline{D}),\left\vert \cdot\right\vert
_{s}\right)  $ is a Banach space, but
\[
C_{0}^{0,s}(\overline{D})\not =\overline{C_{c}^{\infty}(D)}^{\left\vert
\cdot\right\vert _{s}}.
\]
That is, $C_{c}^{\infty}(D)$ is not $\left\vert \cdot\right\vert _{s}$-dense
in $C_{0}^{0,s}(\overline{D}).$

However, we have the following lemma that follows from \cite[Lemma 9]{EPS}.

\begin{lemma}
\label{dens}Let $v\in C_{0}^{0,s}(\overline{D}).$ There exists a sequence
$\left\{  v_{k}\right\}  \subset C_{c}^{\infty}(D)$ such that%
\[
\lim_{k\rightarrow\infty}\left\Vert v_{k}\right\Vert _{\infty}=\left\Vert
v\right\Vert _{\infty}\quad\mathrm{and}\quad\limsup_{k\rightarrow\infty
}\left\vert v_{k}\right\vert _{s}\leq\left\vert v\right\vert _{s}.
\]

\end{lemma}

Now, returning to our bounded domain $\Omega,$ let%
\[
R:=\max_{x\in\overline{\Omega}}\operatorname{dist}(x,\mathbb{R}^{N}%
\setminus\Omega).
\]
It is the inradius of $%
\Omega
$: the radius of the largest ball inscribed in $%
\Omega
.$

Let $B_{R}(x_{0})$ be a ball centered at $x_{0}\in\Omega$ with radius $R$ and
let $\phi_{R}:\overline{B_{R}(x_{0})}\rightarrow\lbrack0,R]$ be the distance
function to the boundary $\partial B_{R}(x_{0}),$ that is,
\[
\phi_{R}(x):=R-\left\vert x-x_{0}\right\vert .
\]

It is simple to verify that $\phi_{R}\in C_{0}^{0,s}(\overline{B_{R}(x_{0}%
)}),$ for every $s\in(0,1],$ with
\begin{equation}
\left\Vert \phi_{R}\right\Vert _{\infty}=R\quad\mathrm{and}\quad\left\vert
\phi_{R}\right\vert _{s}=R^{1-s}. \label{sPhi}%
\end{equation}
Moreover, it is clear that $\phi_{R}$ extended by zero outside $B_{R}(x_{0})$
belongs to $C_{0}^{0,s}(\overline{\Omega})$ and its $s$-H\"{o}lder seminorm is
preserved. In particular, such an extension is a Lipschitz function vanishing
outside $\Omega.$ Hence,
\[
\phi_{R}\in W_{0}^{1,m}(\Omega)\hookrightarrow W_{0}^{s,m}(\Omega)\quad
\forall\,s\in(0,1)\quad\mathrm{and}\quad m\geq1.
\]
(Note that we are considering $\Omega$ at least a Lipschitz domain.)
Consequently, we can apply \cite[Lemma 7]{EPS} to conclude that%
\begin{equation}
\lim_{m\rightarrow\infty}\left[  \phi_{R}\right]  _{s,m}=\left\vert \phi
_{R}\right\vert _{s}=R^{1-s},\quad\mathrm{for}\,\mathrm{each}\,s\in(0,1).
\label{s-holder}%
\end{equation}

The proof of the following proposition is adapted from \cite{LL} where
(\ref{aux2c}) is proved.

\begin{proposition}
\label{Rmin}For each $s\in(0,1]$ one has%
\begin{equation}
\lim_{m\rightarrow\infty}\sqrt[m]{\lambda_{s,m}}=R^{-s}=\frac{\left\vert
\phi_{R}\right\vert _{s}}{\left\Vert \phi_{R}\right\Vert _{\infty}}=\min_{v\in
C_{0}^{0,s}(\overline{\Omega})\setminus\left\{  0\right\}  }\frac{\left\vert
v\right\vert _{s}}{\left\Vert v\right\Vert _{\infty}}. \label{linds}%
\end{equation}

\end{proposition}

\begin{proof}
The second equality in (\ref{linds}) follows from (\ref{sPhi}). Since
$\phi_{R}\in C_{0}^{0,s}(\overline{\Omega})\setminus\left\{  0\right\}  $ to
prove the third equality in (\ref{linds}) it suffices to verify that
\begin{equation}
R^{-s}\leq\frac{\left\vert v\right\vert _{s}}{\left\Vert v\right\Vert
_{\infty}}\quad\forall\,v\in C_{0}^{0,s}(\overline{\Omega})\setminus\left\{
0\right\}  . \label{aux2f}%
\end{equation}

Let $v\in C_{0}^{0,s}(\overline{\Omega})\setminus\left\{  0\right\}  .$
According to Lemma \ref{dens}, there exists a sequence $\left\{
v_{k}\right\}  \subset C_{c}^{\infty}(\Omega)$ such that%
\[
\lim_{k\rightarrow\infty}\left\Vert v_{k}\right\Vert _{\infty}=\left\Vert
v\right\Vert _{\infty}\quad\mathrm{and}\quad\limsup_{k\rightarrow\infty
}\left\vert v_{k}\right\vert _{s}\leq\left\vert v\right\vert _{s}.
\]
Hence, (\ref{aux2c}) yields%
\[
R^{-s}\leq\limsup_{k\rightarrow\infty}\frac{\left\vert v_{k}\right\vert _{s}%
}{\left\Vert v_{k}\right\Vert _{\infty}}\leq\frac{\left\vert v\right\vert
_{s}}{\left\Vert v\right\Vert _{\infty}},
\]
concluding the proof of the third equality in (\ref{linds})

Now, let us prove that
\[
\lim_{m\rightarrow\infty}\sqrt[m]{\lambda_{s,m}}=R^{-s}.
\]

First, observing that%
\[
\sqrt[m]{\lambda_{s,m}}\leq\frac{\left[  \phi_{R}\right]  _{s,m}}{\left\Vert
\phi_{R}\right\Vert _{\infty}}%
\]
we obtain from (\ref{sPhi}) and (\ref{s-holder}) that%
\begin{equation}
\limsup_{m\rightarrow\infty}\sqrt[m]{\lambda_{s,m}}\leq\lim_{m\rightarrow
\infty}\frac{\left[  \phi_{R}\right]  _{s,m}}{\left\Vert \phi_{R}\right\Vert
_{\infty}}=\frac{\left\vert \phi_{R}\right\vert _{s}}{\left\Vert \phi
_{R}\right\Vert _{\infty}}=R^{-s}. \label{aux2d}%
\end{equation}

To prove that
\begin{equation}
R^{-s}\leq\liminf_{m\rightarrow\infty}\left(  \sqrt[m]{\lambda_{s,m}}\right)
\label{aux2e}%
\end{equation}
we fix $m_{0}>\frac{N}{s}$ and take, for each $m$ sufficiently large,
$u_{m}\in W_{0}^{s,m}(\Omega)$ such that $\left\Vert u_{m}\right\Vert
_{\infty}=1$ and%
\[
\lambda_{s,m}=\left[  u_{m}\right]  _{s,m}^{m}.
\]
According to (\ref{Morrey}), we have
\begin{align*}
\left\vert u_{m}\right\vert _{s-\frac{N}{m_{0}}}  &  =\sup_{(x,y)\not =%
(0,0)}\frac{\left\vert u_{m}(x)-u_{m}(y)\right\vert }{\left\vert
x-y\right\vert ^{s-\frac{N}{m_{0}}}}\\
&  =\sup_{(x,y)\not =(0,0)}\frac{\left\vert u_{m}(x)-u_{m}(y)\right\vert
}{\left\vert x-y\right\vert ^{s-\frac{N}{m}}}\left\vert x-y\right\vert
^{(\frac{N}{m_{0}}-\frac{N}{m})}\\
&  \leq(\operatorname{diam}(\Omega))^{(\frac{N}{m_{0}}-\frac{N}{m})}%
\sup_{(x,y)\not =(0,0)}\frac{\left\vert u_{m}(x)-u_{m}(y)\right\vert
}{\left\vert x-y\right\vert ^{s-\frac{N}{m}}}\\
&  \leq(\operatorname{diam}(\Omega))^{(\frac{N}{m_{0}}-\frac{N}{m})}C\left[
u_{m}\right]  _{s,m}=C(\operatorname{diam}(\Omega))^{(\frac{N}{m_{0}}-\frac
{N}{m})}\sqrt[m]{\lambda_{s,m}}.
\end{align*}

The estimate (\ref{aux2d}) implies that $\left\{  u_{m}\right\}  $ is
uniformly bounded in the H\"{o}lder space $C_{0}^{0,\beta-\frac{N}{m_{0}}%
}(\overline{\Omega}),$ which is compactly embedded in $C_{0}(\overline{\Omega
}).$ It follows that, up to a subsequence, $\left\{  u_{m}\right\}  $
converges uniformly in $\overline{\Omega}$ to a function $u\in C_{0}%
(\overline{\Omega})$ such that $\left\Vert u\right\Vert _{\infty}=1.$

For each $1<k<m,$ we have, by H\"{o}lder's inequality,%
\begin{align*}
\int_{\Omega}\int_{\Omega}\frac{\left\vert u_{m}(x)-u_{m}(y)\right\vert ^{k}%
}{\left\vert x-y\right\vert ^{(\frac{N}{m}+s)k}}\mathrm{d}x\mathrm{d}y  &
\leq\left\vert \Omega\right\vert ^{2(1-\frac{k}{m})}\left(  \int_{\Omega}%
\int_{\Omega}\frac{\left\vert u_{m}(x)-u_{m}(y)\right\vert ^{m}}{\left\vert
x-y\right\vert ^{N+sm}}\mathrm{d}x\mathrm{d}y\right)  ^{\frac{k}{m}}\\
&  \leq\left\vert \Omega\right\vert ^{2(1-\frac{k}{m})}\left(  \left[
u_{m}\right]  _{s,m}\right)  ^{k}=\left\vert \Omega\right\vert ^{2(1-\frac
{k}{m})}\left(  \sqrt[m]{\lambda_{s,m}}\right)  ^{k}.
\end{align*}
Making $m\rightarrow\infty$, using the uniform convergence, Fatou's Lemma and
the above estimate we obtain%
\begin{align*}
\int_{\Omega}\int_{\Omega}\frac{\left\vert u(x)-u(y)\right\vert ^{k}%
}{\left\vert x-y\right\vert ^{sk}}\mathrm{d}x\mathrm{d}y  &  \leq
\liminf_{m\rightarrow\infty}\int_{\Omega}\int_{\Omega}\frac{\left\vert
u_{m}(x)-u_{m}(y)\right\vert ^{k}}{\left\vert x-y\right\vert ^{\frac{Nk}%
{m}+sk}}\mathrm{d}x\mathrm{d}y\\
&  \leq\left\vert \Omega\right\vert ^{2}\liminf_{m\rightarrow\infty}\left(
\sqrt[m]{\lambda_{s,m}}\right)  ^{k}.
\end{align*}

Therefore,
\[
\left\vert u\right\vert _{s}=\lim_{k\rightarrow\infty}\left(  \int_{\Omega
}\int_{\Omega}\frac{\left\vert u(x)-u(y)\right\vert ^{k}}{\left\vert
x-y\right\vert ^{sk}}\mathrm{d}x\mathrm{d}y\right)  ^{\frac{1}{k}}\leq
\liminf_{m\rightarrow\infty}\left(  \sqrt[m]{\lambda_{s,m}}\right)  .
\]
Since $R^{-s}\leq\left\vert u\right\vert _{s}$ (according to (\ref{aux2f})) we
obtain (\ref{aux2e}).
\end{proof}

In the remaining of this section we fix $\alpha,\beta\in(0,1),$ with
$\alpha\not =\beta,$ and consider $q$ a continuous function of $p$ satisfying
\begin{equation}
\lim_{p\rightarrow\infty}\frac{q}{p}=:Q\in\left\{
\begin{array}
[c]{lll}%
(0,1) & \mathrm{if} & 0<\beta<\alpha<1\\
(1,\infty) & \mathrm{if} & 0<\alpha<\beta<1.
\end{array}
\right.  \label{q/p}%
\end{equation}
We maintain the notation $q$ instead of $q(p)$ to simplify the presentation.
Note that (\ref{q/p}) implies that
\[
\lim_{p\rightarrow\infty}q=\infty.
\]
Moreover, $q<p$ if $Q\in(0,1)$ and $p<q$ if $Q\in(1,\infty).$

Our goal is to study the asymptotic behavior, as $p\rightarrow\infty,$ of the
least energy solution $u_{p}$ of the problem%
\begin{equation}
\left\Vert u\right\Vert _{\infty}=u(x_{p})\quad\mathrm{and}\quad\left\{
\begin{array}
[c]{lll}%
\left[  \left(  -\Delta_{p}\right)  ^{\alpha}+\left(  -\Delta_{q}\right)
^{\beta}\right]  u=\mu_{p}\left\Vert u\right\Vert _{\infty}^{p-1}\delta
_{x_{p}} & \mathrm{in} & \Omega\\
u>0 & \mathrm{in} & \Omega\\
u=0 & \mathrm{in} & \mathbb{R}^{N}\setminus\Omega,
\end{array}
\right.  \label{pq(p)}%
\end{equation}
where $\mu_{p}$ satisfies%
\begin{equation}
\Lambda:=\lim_{p\rightarrow\infty}\sqrt[p]{\mu_{p}}>R^{-\alpha}, \label{Lamp}%
\end{equation}
with $R$ denoting the inradius of $\Omega.$

This condition guarantees that%
\begin{equation}
\mu_{p}>\lambda_{\alpha,p} \label{lambp}%
\end{equation}
for all $p$ sufficiently large, say $p>p_{0}$. Moreover, by taking a larger
$p_{0}$ one of the conditions (\ref{H1b}) or (\ref{H1a}) is fulfilled. So,
according to Theorem \ref{teo1}, for each $p>p_{0}$ the problem (\ref{pq(p)}%
)\ has at least one positive least energy solution
\[
u_{p}\in X_{p}(\Omega):=\left\{
\begin{array}
[c]{lll}%
W_{0}^{\alpha,p}(\Omega) & \mathrm{if} & 0<\beta<\alpha<1\\
W_{0}^{\beta,q}(\Omega) & \mathrm{if} & 0<\alpha<\beta<1.
\end{array}
\right.
\]

\begin{remark}
\label{cdneh}Combining (\ref{s-holder}) and (\ref{Lamp}) we have
\[
\lim_{p\rightarrow\infty}\frac{\left[  \phi_{R}\right]  _{\alpha,p}%
}{\left\Vert \phi_{R}\right\Vert _{\infty}}=R^{-\alpha}<\Lambda:=\lim
_{p\rightarrow\infty}\sqrt[p]{\mu_{p}}.
\]
Consequently, $\mu_{p}\left\Vert \phi_{R}\right\Vert _{\infty}^{p}>\left[
\phi_{R}\right]  _{\alpha,p}^{p}$ for all $p$ large enough.
\end{remark}

\begin{proposition}
\label{limits}Suppose (\ref{q/p}) and (\ref{Lamp}) hold. Then,%
\begin{equation}
\lim_{p\rightarrow\infty}\left[  u_{p}\right]  _{\beta,q}=\left(  \Lambda
R^{\beta}\right)  ^{\frac{1}{Q-1}}\quad\mathrm{and}\quad\lim_{p\rightarrow
\infty}\left\Vert u_{p}\right\Vert _{\infty}=R^{\beta}(\Lambda R^{\beta
})^{\frac{1}{Q-1}}, \label{lim1b}%
\end{equation}

\end{proposition}

\begin{proof}
We assume that $p$ is large enough so that $u_{p}$ exists according to Theorem
\ref{teo1}.

Since $u_{p}$ is a weak solution of (\ref{pq(p)}) and $W_{0}^{\beta,q}%
(\Omega)$ is continuously embedded into $C(\overline{\Omega})$ we have%
\begin{equation}
\left[  u_{p}\right]  _{\beta,q}^{q}\leq\left[  u_{p}\right]  _{\alpha,p}%
^{p}+\left[  u_{p}\right]  _{\beta,q}^{q}=\mu_{p}\left\Vert u_{p}\right\Vert
_{\infty}^{p}\leq\mu_{p}\frac{\left[  u_{p}\right]  _{\beta,q}^{p}}%
{(\sqrt[q]{\lambda_{\beta,q}})^{p}}, \label{aux1b}%
\end{equation}
so that%
\[
\frac{(\left[  u_{p}\right]  _{\beta,q})^{\frac{q}{p}}}{\sqrt[p]{\mu_{p}}}%
\leq\left\Vert u_{p}\right\Vert _{\infty}\leq\frac{\left[  u_{p}\right]
_{\beta,q}}{\sqrt[q]{\lambda_{\beta,q}}}.
\]

Hence, taking into account the first equality in (\ref{linds}) and
(\ref{Lamp}) we easily check that the second limit in (\ref{lim1b}) is
consequence of the first one.

Let us then prove the first limit (\ref{lim1b}).

We start with the case $Q\in(1,\infty),$ where necessarily $p<q$ (and
$0<\alpha<\beta$). After isolating $\left[  u_{p}\right]  _{\beta,q}$ in
(\ref{aux1b}) we obtain%
\begin{equation}
\limsup_{p\rightarrow\infty}\left[  u_{p}\right]  _{\beta,q}\leq
\lim_{p\rightarrow\infty}\left(  \frac{\sqrt[p]{\mu_{p}}}{\sqrt[q]%
{\lambda_{\beta,q}}}\right)  ^{\frac{p}{q-p}}=\left(  \Lambda R^{\beta
}\right)  ^{\frac{1}{Q-1}}. \label{aux2b}%
\end{equation}

Let%
\[
t=\left(  \frac{\mu_{p}\left\Vert \phi_{R}\right\Vert _{\infty}^{p}-\left[
\phi_{R}\right]  _{\alpha,p}^{p}}{\left[  \phi_{R}\right]  _{\beta,q}^{q}%
}\right)  ^{\frac{1}{q-p}}=\left(  \frac{\mu_{p}R^{p}-\left[  \phi_{R}\right]
_{\alpha,p}^{p}}{\left[  \phi_{R}\right]  _{\beta,q}^{q}}\right)  ^{\frac
{1}{q-p}}.
\]
(Note from Remark \ref{cdneh} that $t$ is well-defined). It is simple to
verify that
\[
E_{\mu_{p}}(t\phi_{R})=\left(  \frac{1}{q}-\frac{1}{p}\right)  t^{q}\left[
\phi_{R}\right]  _{\beta,q}^{q}.
\]

Noticing that
\[
\left(  \frac{1}{q}-\frac{1}{p}\right)  \left[  u_{p}\right]  _{\beta,q}%
^{q}=E_{\mu_{p}}(u_{p})\leq E_{\mu_{p}}(t\phi_{R})=\left(  \frac{1}{q}%
-\frac{1}{p}\right)  t^{q}\left[  \phi_{R}\right]  _{\beta,q}^{q}<0
\]
we obtain
\[
\left[  u_{p}\right]  _{\beta,q}\geq t\left[  \phi_{R}\right]  _{\beta
,q}=\left(  \frac{\mu_{p}R^{p}-\left[  \phi_{R}\right]  _{\alpha,p}^{p}%
}{\left[  \phi_{R}\right]  _{\beta,q}^{q}}\right)  ^{\frac{1}{q-p}}\left[
\phi_{R}\right]  _{\beta,q}=\left(  \frac{\sqrt[p]{\mu_{p}}R}{\left[  \phi
_{R}\right]  _{\beta,q}}\sqrt[p]{1-(a_{p})^{p}}\right)  ^{\frac{p}{q-p}},
\]
where%
\begin{equation}
a_{p}:=\frac{\left[  \phi_{R}\right]  _{\alpha,p}}{\sqrt[p]{\mu_{p}}R}.
\label{ap}%
\end{equation}

Since%
\[
\lim_{p\rightarrow\infty}a_{p}=\frac{R^{1-\alpha}}{\Lambda R}=\frac
{R^{-\alpha}}{\Lambda}<1,
\]
we can verify that%
\[
\lim_{p\rightarrow\infty}\sqrt[p]{1-(a_{p})^{p}}=1.
\]
Hence,
\[
\liminf_{p\rightarrow\infty}\left[  u_{p}\right]  _{\beta,q}\geq
\lim_{p\rightarrow\infty}\left(  \frac{\sqrt[p]{\mu_{p}}R}{\left[  \phi
_{R}\right]  _{\beta,q}}\right)  ^{\frac{p}{q-p}}\lim_{p\rightarrow\infty
}\left(  \sqrt[p]{1-(a_{p})^{p}}\right)  ^{\frac{p}{q-p}}=\left(
\frac{\Lambda R}{R^{1-\beta}}\right)  ^{\frac{1}{Q-1}}=(\Lambda R^{\beta
})^{\frac{1}{Q-1}}.
\]

Combining this with (\ref{aux2b}) we obtain the first limit in (\ref{lim1b}).

Now, let us analyze the case $Q\in(0,1),$ where necessarily $q<p$ (and
$0<\beta<\alpha$). In this case,
\[
0<\left(  \frac{1}{q}-\frac{1}{p}\right)  \left[  u_{p}\right]  _{\beta,q}%
^{q}=E_{\mu_{p}}(u_{p})\leq E_{\mu_{p}}(t\phi_{R})=\left(  \frac{1}{q}%
-\frac{1}{p}\right)  t^{q}\left[  \phi_{R}\right]  _{\beta,q}^{q}%
\]
where%
\[
t=\left(  \frac{\left[  \phi_{R}\right]  _{\beta,q}^{q}}{\mu_{p}\left\Vert
\phi_{R}\right\Vert _{\infty}^{p}-\left[  \phi_{R}\right]  _{\alpha,p}^{p}%
}\right)  ^{\frac{1}{p-q}}=\left(  \frac{\left[  \phi_{R}\right]  _{\beta
,q}^{q}}{\mu_{p}R^{p}-\left[  \phi_{R}\right]  _{\alpha,p}^{p}}\right)
^{\frac{1}{p-q}}%
\]
(which is also well-defined according to Remark \ref{cdneh}). It follows that%
\[
\left[  u_{p}\right]  _{\beta,q}\leq t\left[  \phi_{R}\right]  _{\beta
,q}=\left(  \frac{\left[  \phi_{R}\right]  _{\beta,q}}{\sqrt[p]{\mu_{p}%
}R\sqrt[p]{1-(a_{p})^{p}}}\right)  ^{\frac{p}{p-q}}%
\]
where $a_{p}$ is also given by (\ref{ap}). Consequently,%
\begin{equation}
\limsup_{p\rightarrow\infty}\left[  u_{p}\right]  _{\beta,q}\leq
\lim_{p\rightarrow\infty}\left(  \frac{\left[  \phi_{R}\right]  _{\beta,q}%
}{\sqrt[p]{\mu_{p}}R\sqrt[p]{1-(a_{p})^{p}}}\right)  ^{\frac{p}{p-q}}=\left(
\frac{R^{1-\beta}}{\Lambda R}\right)  ^{\frac{1}{1-Q}}=\left(  \Lambda
R^{\beta}\right)  ^{\frac{1}{Q-1}}. \label{aux2a}%
\end{equation}

After isolating $\left[  u_{p}\right]  _{\beta,q}$ in (\ref{aux1b}) we obtain%
\[
\liminf_{p\rightarrow\infty}\left[  u_{p}\right]  _{\beta,q}\geq
\lim_{p\rightarrow\infty}\left(  \frac{\sqrt[q]{\lambda_{\beta,q}}}%
{\sqrt[p]{\mu_{p}}}\right)  ^{\frac{p}{p-q}}=\left(  \frac{R^{-\beta}}%
{\Lambda}\right)  ^{\frac{1}{1-Q}}=\left(  \Lambda R^{\beta}\right)
^{\frac{1}{Q-1}},
\]
which combined with (\ref{aux2a}) provides the first limit in (\ref{lim1b}).
\end{proof}

\begin{corollary}
\label{cor1}Suppose (\ref{q/p}) and (\ref{Lamp}) hold. Then,%
\[
\frac{1}{\Lambda R^{\alpha}}(\Lambda R^{\beta})^{\frac{Q}{Q-1}}\leq
\liminf_{p\rightarrow\infty}\left[  u_{p}\right]  _{\alpha,p}\leq
\limsup_{p\rightarrow\infty}\left[  u_{p}\right]  _{\alpha,p}\leq(\Lambda
R^{\beta})^{\frac{Q}{Q-1}}.
\]

\end{corollary}

\begin{proof}
It follows from the second limit in (\ref{lim1b}) combined with the estimates%
\[
\lambda_{\alpha,p}\left\Vert u_{p}\right\Vert _{\infty}^{p}\leq\left[
u_{p}\right]  _{\alpha,p}^{p}\leq\left[  u_{p}\right]  _{\alpha,p}^{p}+\left[
u_{p}\right]  _{\beta,q}^{q}=\mu_{p}\left\Vert u_{p}\right\Vert _{\infty}%
^{p}.
\]

\end{proof}

In the next proposition we prove that the limit functions of the family
$\left\{  u_{p}\right\}  _{p>p_{0}},$ as $p\rightarrow\infty,$ belongs to
$C_{0}^{0,\beta}(\overline{\Omega})$ and minimize the quotient $\left\vert
u\right\vert _{\beta}/\left\Vert u\right\Vert _{\infty}$ in $C_{0}^{0,\beta
}(\overline{\Omega})\setminus\left\{  0\right\}  .$

\begin{proposition}
\label{betaHolder}Let $\left\{  p_{n}\right\}  $ and $\left\{  q_{n}\right\}
$ satisfying (\ref{q/p}), with $p_{n}\rightarrow\infty,$ and let $\mu_{n}%
:=\mu_{p_{n}}$ satisfying (\ref{Lamp}). Then, there exist $u_{\infty}\in
C_{0}^{0,\beta}(\overline{\Omega})$ and $x_{\infty}\in\Omega$ such that, up to
subsequences, $u_{p_{n}}\rightarrow u_{\infty}$ uniformly in $\overline
{\Omega}$ and $x_{p_{n}}\rightarrow x_{\infty}\in\Omega,$ with
\[
u_{\infty}(x_{\infty})=\left\Vert u_{\infty}\right\Vert _{\infty}=R^{\beta
}(\Lambda R^{\beta})^{\frac{1}{Q-1}}.
\]
Moreover,%
\[
\left\vert u_{\infty}\right\vert _{\beta}=(\Lambda R^{\beta})^{\frac{1}{Q-1}%
}=\lim\left[  u_{p_{n}}\right]  _{\beta,q_{n}}%
\]
and%
\begin{equation}
\frac{\left\vert u_{\infty}\right\vert _{\beta}}{\left\Vert u_{\infty
}\right\Vert _{\infty}}=\frac{1}{R^{\beta}}=\min_{u\in C_{0}^{0,\beta
}(\overline{\Omega})\setminus\left\{  0\right\}  }\frac{\left\vert
u\right\vert _{\beta}}{\left\Vert u\right\Vert _{\infty}}. \label{betaholder}%
\end{equation}

\end{proposition}

\begin{proof}
Since $\Omega$ is bounded, we can assume that (passing to a subsequence)
$x_{p_{n}}$ converges to a point $x_{\infty}\in\overline{\Omega}.$ Fix
$m_{0}>N/\beta\ $and assume that $n$ is large enough so that $m_{0}<\left\{
p_{n},q_{n}\right\}  .$

Taking into account the inequality (\ref{Morrey}), we have (as in Proposition
\ref{Rmin})
\begin{align*}
\left\vert u_{p_{n}}\right\vert _{\beta-\frac{N}{m_{0}}}  &  =\sup
_{(x,y)\not =(0,0)}\frac{\left\vert u_{p_{n}}(x)-u_{p_{n}}(y)\right\vert
}{\left\vert x-y\right\vert ^{\beta-\frac{N}{m_{0}}}}\\
&  \leq(\operatorname{diam}(\Omega))^{(\frac{N}{m_{0}}-\frac{N}{q_{n}})}%
\sup_{(x,y)\not =(0,0)}\frac{\left\vert u_{p_{n}}(x)-u_{p_{n}}(y)\right\vert
}{\left\vert x-y\right\vert ^{\beta-\frac{N}{q_{n}}}}\leq(\operatorname{diam}%
(\Omega))^{(\frac{N}{m_{0}}-\frac{N}{q_{n}})}C\left[  u_{p_{n}}\right]
_{\beta,q_{n}}.
\end{align*}

The first limit in (\ref{lim1b}) implies that $\left\{  u_{p_{n}}\right\}  $
is uniformly bounded in the H\"{o}lder space $C_{0}^{0,\beta-\frac{N}{m_{0}}%
}(\overline{\Omega}),$ which is compactly embedded in $C_{0}(\overline{\Omega
}).$ It follows that, up to a subsequence, $\left\{  u_{p_{n}}\right\}  $
converges uniformly in $\overline{\Omega}$ to a function $u_{\infty}\in
C_{0}(\overline{\Omega}).$ Of course, $\left\Vert u_{\infty}\right\Vert
=u_{\infty}(x_{\infty})$ and, by virtue of the second limit in (\ref{lim1b}),
\[
u_{\infty}(x_{\infty})=R^{\beta}(\Lambda R^{\beta})^{\frac{1}{Q-1}}>0,
\]
so that $x_{\infty}\not \in \partial\Omega.$

Now, if $m>m_{0}$ and $n$ is sufficiently large such that $q_{n}>m$,
H\"{o}lder's inequality yields
\begin{align*}
\int_{\Omega}\int_{\Omega}\frac{\left\vert u_{p_{n}}(x)-u_{p_{n}%
}(y)\right\vert ^{m}}{\left\vert x-y\right\vert ^{\frac{Nm}{q_{n}}+\beta m}%
}\mathrm{d}x\mathrm{d}y  &  \leq\left\vert \Omega\right\vert ^{2(1-\frac
{m}{q_{n}})}\left(  \int_{\Omega}\int_{\Omega}\frac{\left\vert u_{p_{n}%
}(x)-u_{p_{n}}(y)\right\vert ^{q_{n}}}{\left\vert x-y\right\vert ^{N+\beta
q_{n}}}\mathrm{d}x\mathrm{d}y\right)  ^{\frac{m}{q_{n}}}\\
&  \leq\left\vert \Omega\right\vert ^{2(1-\frac{m}{q_{n}})}\left(  \left[
u_{p_{n}}\right]  _{\beta,q_{n}}\right)  ^{m}.
\end{align*}

Hence, combining the first limit in (\ref{lim1b}) and Fatou's Lemma,%
\begin{align*}
\int_{\Omega}\int_{\Omega}\frac{\left\vert u_{\infty}(x)-u_{\infty
}(y)\right\vert ^{m}}{\left\vert x-y\right\vert ^{\beta m}}\mathrm{d}%
x\mathrm{d}y  &  \leq\liminf_{n\rightarrow\infty}\int_{\Omega}\int_{\Omega
}\frac{\left\vert u_{p_{n}}(x)-u_{p_{n}}(y)\right\vert ^{m}}{\left\vert
x-y\right\vert ^{\frac{Nm}{q_{n}}+\beta m}}\mathrm{d}x\mathrm{d}y\\
&  \leq\left\vert \Omega\right\vert ^{2}\liminf_{n\rightarrow\infty}\left(
\left[  u_{p_{n}}\right]  _{\beta,q_{n}}\right)  ^{m}=\left\vert
\Omega\right\vert ^{2}\left(  \Lambda R^{\beta}\right)  ^{\frac{m}{Q-1}}.
\end{align*}

Therefore,
\[
\left\vert u_{\infty}\right\vert _{\beta}=\lim_{m\rightarrow\infty}\left(
\int_{\Omega}\int_{\Omega}\frac{\left\vert u_{\infty}(x)-u_{\infty
}(y)\right\vert ^{m}}{\left\vert x-y\right\vert ^{\beta m}}\mathrm{d}%
x\mathrm{d}y\right)  ^{\frac{1}{m}}\leq(\Lambda R^{\beta})^{\frac{1}{Q-1}}.
\]

It follows that $u_{\infty}\in C_{0}^{0,\beta}(\overline{\Omega}).$ Hence,
observing that
\[
\frac{1}{R^{\beta}}=\min_{v\in C_{0}^{0,\beta}(\overline{\Omega}%
)\setminus\left\{  0\right\}  }\frac{\left\vert v\right\vert _{\beta}%
}{\left\Vert v\right\Vert _{\infty}}\leq\frac{\left\vert u_{\infty}\right\vert
_{\beta}}{\left\Vert u_{\infty}\right\Vert _{\infty}}=\frac{\left\vert
u_{\infty}\right\vert _{\beta}}{R^{\beta}(\Lambda R^{\beta})^{\frac{1}{Q-1}}%
},
\]
we obtain%
\[
\left\vert u_{\infty}\right\vert _{\beta}\geq(\Lambda R^{\beta})^{\frac
{1}{Q-1}}.
\]

Therefore,
\[
\left\vert u_{\infty}\right\vert _{\beta}=(\Lambda R^{\beta})^{\frac{1}{Q-1}%
}=\lim\left[  u_{p_{n}}\right]  _{\beta,q_{n}}%
\]
and
\[
\frac{\left\vert u_{\infty}\right\vert _{\beta}}{\left\Vert u_{\infty
}\right\Vert _{\infty}}=\frac{\left(  \Lambda R^{\beta}\right)  ^{\frac
{1}{Q-1}}}{R^{\beta}(\Lambda R^{\beta})^{\frac{1}{Q-1}}}=\frac{1}{R^{\beta}%
}=\min_{v\in C_{0}^{0,\beta}(\overline{\Omega})\setminus\left\{  0\right\}
}\frac{\left\vert v\right\vert _{\beta}}{\left\Vert v\right\Vert _{\infty}}.
\]

\end{proof}

\begin{remark}
Considering Corollary \ref{cor1} we can reproduce the proof of Proposition
\ref{betaHolder} to conclude that, in the case $Q\in(0,1),$ the limit function
is more regular: $u_{\infty}\in C_{0}^{0,\alpha}(\overline{\Omega})$ and,
moreover,
\[
R^{-\alpha}\leq\frac{\left\vert u_{\infty}\right\vert _{\alpha}}{\left\Vert
u_{\infty}\right\Vert _{\infty}}\leq\Lambda.
\]
These estimates are also valid in the complementary case $Q\in(1,\infty),$
where obviously the $\beta$-regularity is better that $\alpha$-regularity
since $0<\alpha<\beta.$
\end{remark}

\begin{corollary}
\label{BH1}One has%
\[
u_{\infty}(x)\leq\left(  \Lambda R^{\beta}\right)  ^{\frac{1}{Q-1}}\left(
\operatorname{dist}(x,\partial\Omega)\right)  ^{\beta}\quad\forall\,x\in\Omega
\]
and, therefore, the maximum point $x_{\infty}$ of $u_{\infty}$ is also a
maximum point of the distance function to the boundary $\partial\Omega.$
\end{corollary}

\begin{proof}
For each $x\in\Omega$ let $y_{x}\in\partial\Omega$ be such
\[
\operatorname{dist}(x,\partial\Omega)=\left\vert x-y_{x}\right\vert .
\]
Then, since $u_{\infty}(y_{x})=0$ and $\left\vert u_{\infty}\right\vert
_{\beta}=\left(  \Lambda R^{\beta}\right)  ^{\frac{1}{Q-1}},$ we get
\[
u_{\infty}(x)=\left\vert u_{\infty}(x)-u_{\infty}(y_{x})\right\vert
\leq\left\vert u_{\infty}\right\vert _{\beta}\left\vert x-y_{x}\right\vert
^{\beta}=\left(  \Lambda R^{\beta}\right)  ^{\frac{1}{Q-1}}\left(
\operatorname{dist}(x,\partial\Omega)\right)  ^{\beta}.
\]
Hence, observing that $\operatorname{dist}(x,\partial\Omega)=\left\vert
x-y_{x}\right\vert \leq R$ and $u_{\infty}(x_{\infty})=R^{\beta}(\Lambda
R^{\beta})^{\frac{1}{Q-1}},$ we obtain%

\[
R^{\beta}(\Lambda R^{\beta})^{\frac{1}{Q-1}}\leq\left(  \Lambda R^{\beta
}\right)  ^{\frac{1}{Q-1}}\left(  \operatorname{dist}(x_{\infty}%
,\partial\Omega)\right)  ^{\beta}\leq\left(  \Lambda R^{\beta}\right)
^{\frac{1}{Q-1}}R^{\beta},
\]
so that%
\[
\operatorname{dist}(x_{\infty},\partial\Omega)=R=\left\Vert
\operatorname{dist}(\cdot,\partial\Omega)\right\Vert _{\infty}.
\]

\end{proof}

In the sequel, we argue that the function $u_{\infty}$ is a viscosity solution
of the equation%
\begin{equation}
\max\left\{  \mathcal{L}_{\alpha}^{+}u,\left(  \mathcal{L}_{\beta}%
^{+}u\right)  ^{Q}\right\}  =\max\left\{  -\mathcal{L}_{\alpha}^{-}u,\left(
-\mathcal{L}_{\beta}^{-}u\right)  ^{Q}\right\}  \label{max=max}%
\end{equation}
in $\Omega\setminus\left\{  x_{\infty}\right\}  $ (the operators
$\mathcal{L}_{\alpha}^{+}$ and $\mathcal{L}_{\alpha}^{-}$ are defined in
(\ref{Ls+-})). This means that $u_{\infty}$ is both a supersolution and a
subsolution of (\ref{max=max}) or, equivalently, $u_{\infty}$ meets the
(respective) requirements:

\begin{itemize}
\item $\max\left\{  \left(  \mathcal{L}_{\alpha}^{+}\varphi\right)
(x_{0}),\left(  \left(  \mathcal{L}_{\beta}^{+}\varphi\right)  (x_{0})\right)
^{Q}\right\}  \leq\max\left\{  \left(  -\mathcal{L}_{\alpha}^{-}%
\varphi\right)  (x_{0}),\left(  \left(  -\mathcal{L}_{\beta}^{-}%
\varphi\right)  (x_{0})\right)  ^{Q}\right\}  $ for every the pair $\left(
x_{0},\varphi\right)  \in(\Omega\setminus\left\{  x_{\infty}\right\}  )\times
C_{c}^{1}(\mathbb{R}^{N})$ satisfying
\[
\varphi(x_{0})=u(x_{0})\quad\mathrm{and}\quad\varphi(x)\leq u(x)\quad
\forall\,x\in\mathbb{R}^{N}\setminus\left\{  x_{0},x_{\infty}\right\}  ,
\]

\item $\max\left\{  \left(  \mathcal{L}_{\alpha}^{+}\varphi\right)
(x_{0}),\left(  \left(  \mathcal{L}_{\beta}^{+}\varphi\right)  (x_{0})\right)
^{Q}\right\}  \geq\max\left\{  \left(  -\mathcal{L}_{\alpha}^{-}%
\varphi\right)  (x_{0}),\left(  \left(  -\mathcal{L}_{\beta}^{-}%
\varphi\right)  (x_{0})\right)  ^{Q}\right\}  $ for every the pair $\left(
x_{0},\varphi\right)  \in(\Omega\setminus\left\{  x_{\infty}\right\}  )\times
C_{c}^{1}(\mathbb{R}^{N})$ satisfying
\[
\varphi(x_{0})=u(x_{0})\quad\mathrm{and}\quad\varphi(x)\geq u(x)\quad
\forall\,x\in\mathbb{R}^{N}\setminus\left\{  x_{0},x_{\infty}\right\}  .
\]

\end{itemize}

A proof of the following result (where $t^{\pm}=\max\left\{  \pm t,0\right\}
$), adapted from \cite[Lemma 6.5]{Chamb}, can be found in \cite[Lemma
6.1]{FPL}.

\begin{lemma}
\label{L+-}If $s\in(0,1),$ $\varphi\in C_{c}^{1}(\mathbb{R}^{N}),$
$\lim\limits_{m\rightarrow\infty}z_{m}\rightarrow x,$ then,%
\[
\lim_{m\rightarrow\infty}A_{m}(\varphi(z_{m}))=(\mathcal{L}_{s,\infty}%
^{+}\varphi)(x_{0})\quad\mathrm{and}\quad\lim_{m\rightarrow\infty}%
B_{m}(\varphi(z_{m}))=(-\mathcal{L}_{s,\infty}^{-}\varphi)(x_{0}),
\]
where%
\[
\left(  A_{m}(\varphi(z_{m}))\right)  ^{m-1}:=2\int_{\mathbb{R}^{N}}%
\dfrac{\left\vert \varphi(y)-\varphi(z_{m})\right\vert ^{m-2}(\varphi
(y)-\varphi(z_{m}))^{+}}{\left\vert y-z_{m}\right\vert ^{N+sm}}\mathrm{d}y.
\]
and%
\[
(B_{m}(\varphi(z_{m})))^{m-1}:=2\int_{\mathbb{R}^{N}}\dfrac{\left\vert
\varphi(y)-\varphi(z_{m})\right\vert ^{m-2}(\varphi(y)-\varphi(z_{m}))^{-}%
}{\left\vert y-z_{m}\right\vert ^{N+sm}}\mathrm{d}y.
\]

\end{lemma}

\begin{proposition}
\label{BH2}The function $u_{\infty}$ is a viscosity solution of (\ref{max=max}%
) in the punctured domain $\Omega\setminus\left\{  x_{\infty}\right\}  .$
Moreover, $u_{\infty}>0$ in $\Omega.$
\end{proposition}

\begin{proof}
We give a sketch of the proof based on \cite{FPL} and \cite{LL}.

In order to verify that $u_{\infty}$ is a supersolution of (\ref{max=max}) in
$\Omega\setminus\left\{  x_{\infty}\right\}  $ we fix a pair $\left(
x_{0},\varphi\right)  \in(\Omega\setminus\left\{  x_{\infty}\right\}  )\times
C_{c}^{1}(\mathbb{R}^{N})$ satisfying
\[
\varphi(x_{0})=u_{\infty}(x_{0})\quad\mathrm{and}\quad\varphi(x)\leq
u_{\infty}(x)\quad\forall\,x\in\mathbb{R}^{N}\setminus\left\{  x_{0}%
,x_{\infty}\right\}  .
\]

Since $x_{0}\not =x_{\infty}=\lim x_{n},$ we can assume that there exist
$n_{0}\in\mathbb{N}$ and a ball $B_{\rho}(x_{0}),$ centered at $x_{0}$ and
with radius $\rho,$ such that
\[
x_{n}\not \in B_{\rho}(x_{0})\subset(\Omega\setminus\left\{  x_{\infty
}\right\}  )\quad\forall\,n\geq n_{0}.
\]

Hence,
\begin{equation}
\mathcal{L}_{\alpha,p_{n}}u_{n}+\mathcal{L}_{\beta,q_{n}}u_{n}=0\quad
\mathrm{in}\,B_{\rho}(x_{0}),\quad\forall\,n\geq n_{0}, \label{unvisc}%
\end{equation}
in the viscosity sense.

By standard arguments, we can construct a sequence $\left\{  z_{n}\right\}
\subset B_{\rho}(x_{0})$ such that $z_{n}\rightarrow x_{0}$ and%
\[
k_{n}:=\min_{B_{\rho}(x_{0})}(u_{n}(x)-\varphi(x))=u_{n}(z_{n})-\varphi
(z_{n})<u_{n}(x)-\varphi(x)\quad\forall\,x\not =x_{n}.
\]

It follows that the function $\psi_{n}:=\varphi+k_{n}$ satisfies
\[
\psi(z_{n})=u_{n}(z_{n})\quad\mathrm{and}\quad\psi(x)<u_{n}(x)\quad
\forall\,x\in B_{\rho}(x_{0}).
\]
Consequently, (see Remark \ref{+c})
\[
(\mathcal{L}_{\alpha,p_{n}}\psi_{n})(z_{n})+(\mathcal{L}_{\beta,q_{n}}\psi
_{n})(z_{n})=(\mathcal{L}_{\alpha,p_{n}}\varphi)(z_{n})+(\mathcal{L}%
_{\beta,q_{n}}\varphi)(z_{n})\leq0,\quad\forall\,n\geq n_{0}.
\]
The inequality can be write as%
\begin{equation}
A_{p_{n}}^{p_{n}-1}+A_{q_{n}}^{q_{n}-1}\leq B_{p_{n}}^{p_{n}-1}+B_{q_{n}%
}^{q_{n}-1} \label{aux3a}%
\end{equation}
where $A_{p_{n}}:=A_{p_{n}}(\varphi(z_{n})),$ $A_{q_{n}}:=A_{q_{n}}%
(\varphi(z_{n})),$ $B_{p_{n}}:=B_{p_{n}}(\varphi(z_{n}))$ and $B_{q_{n}%
}:=B_{q_{n}}(\varphi(z_{n})).$

We have
\begin{align*}
\lim\sqrt[p_{n}-1]{A_{p_{n}}^{p_{n}-1}+A_{q_{n}}^{q_{n}-1}}  &  =\lim
A_{p_{n}}\left(  \sqrt[p_{n}-1]{1+\left(  (A_{q_{n}})^{\frac{q_{n}-1}{p_{n}%
-1}}/A_{p_{n}}\right)  ^{p_{n}-1}}\right) \\
&  =\max\left\{  \lim A_{p_{n}},(\lim A_{q_{n}})^{\lim\frac{q_{n}-1}{p_{n-1}}%
}\right\} \\
&  =\max\left\{  (\mathcal{L}_{\alpha}^{+}\varphi)(x_{0}),\left[  \left(
\mathcal{L}_{\beta}^{+}\varphi\right)  (x_{0})\right]  ^{Q}\right\}  ,
\end{align*}
where the latter equality follows from Lemma \ref{L+-}. Analogously, we
compute
\[
\lim\sqrt[p_{n}-1]{B_{p_{n}}^{p_{n}-1}+B_{q_{n}}^{q_{n}-1}}=\max\left\{
(-\mathcal{L}_{\alpha}^{-}\varphi)(x_{0}),\left[  \left(  -\mathcal{L}_{\beta
}^{-}\varphi\right)  (x_{0})\right]  ^{Q}\right\}  .
\]

Therefore, (\ref{aux3a}) yields%
\[
\max\left\{  (\mathcal{L}_{\alpha}^{+}\varphi)(x_{0}),\left[  \left(
\mathcal{L}_{\beta}^{+}u\right)  (x_{0})\right]  ^{Q}\right\}  \leq
\max\left\{  (-\mathcal{L}_{\alpha}^{-}\varphi)(x_{0}),\left[  \left(
-\mathcal{L}_{\beta}^{-}u\right)  (x_{0})\right]  ^{Q}\right\}  ,
\]
which shows that $u_{\infty}$ is a viscosity supersolution of (\ref{max=max})
in $\Omega\setminus\left\{  x_{\infty}\right\}  .$

Similarly, by symmetric arguments, we can prove that $u_{\infty}$ is a
viscosity subsolution of (\ref{max=max}) in $\Omega\setminus\left\{
x_{\infty}\right\}  .$

The positivity of $u_{\infty}$ in $\Omega$ comes from the fact that
$u_{\infty}\ $is a supersolution of (\ref{max=max}). Indeed, adapting the
argument of \cite[Lemma 22]{LL}, if $u_{\infty}(x_{0})=0,$ then either
\[
(\mathcal{L}_{\alpha}^{+}\varphi)(x_{0})\leq(-\mathcal{L}_{\alpha}^{-}%
\varphi)(x_{0})\quad\mathrm{or}\quad(\mathcal{L}_{\beta}^{+}\varphi
)(x_{0})\leq(-\mathcal{L}_{\beta}^{-}\varphi)(x_{0})
\]
for a nonnegative, nontrivial $\varphi\in C_{c}^{1}(\mathbb{R}^{N})$
satisfying%
\[
\varphi(x_{0})=u_{\infty}(x_{0})=0\leq\varphi(x)\leq u_{\infty}(x)\quad
\forall\,x\in\mathbb{R}^{N}\setminus\left\{  x_{0},x_{\infty}\right\}  .
\]
In the first case, this yields
\[
\frac{\varphi(z)}{\left\vert x_{0}-y\right\vert ^{\alpha}}\leq\sup
_{y\in\mathbb{R}^{N}\setminus\left\{  x_{0}\right\}  }\frac{\varphi
(y)}{\left\vert x_{0}-y\right\vert ^{\alpha}}+\inf_{y\in\mathbb{R}%
^{N}\setminus\left\{  x_{0}\right\}  }\frac{\varphi(y)}{\left\vert
x_{0}-y\right\vert ^{\alpha}}\leq0\quad\forall\,z\in\mathbb{R}^{N}%
\setminus\left\{  x_{0}\right\}
\]
and leads to the contradiction $\varphi\equiv0.$ Obviously, in the second case
we arrive at the same contradiction.
\end{proof}

\begin{proof}
[Proof of Theorem \ref{Main}]It follows by gathering Proposition
\ref{betaHolder}, Corollary \ref{BH1} and Proposition \ref{BH2}.
\end{proof}

\section{Acknowledgements}

G. Ercole was partially supported by CNPq/Brazil (306815/2017-6 and
422806/2018-8) and Fapemig/Brazil (CEX-PPM-00137-18).

\end{document}